\documentclass[a4paper, 11pt]{amsart}
\usepackage{amsmath, amsthm, amscd, amssymb, amsfonts, amsxtra, amssymb, mathtools, latexsym}
\usepackage{enumerate}
\usepackage{verbatim}
\usepackage{textcomp}
\usepackage{fancybox, float, shadow, color, graphicx}

\newcommand{\ff}{\mathbb{F}}

\newcommand{\cod}{\mathcal{C}}
\newcommand{\tr}{\operatorname{Tr}}

\newcommand{\CC}{\mathcal{C}}
\newcommand{\LL}{\mathcal{L}}

\newtheorem{thm}{Theorem}[section]
\newtheorem{prop}[thm]{Proposition}
\newtheorem{lem}[thm]{Lemma}
\newtheorem{coro}[thm]{Corollary}
\theoremstyle{definition}
\newtheorem{rem}[thm]{Remark}
\newtheorem{exam}[thm]{Example}
\newtheorem{defi}[thm]{Definition}

\theoremstyle{remark}

\usepackage{color}
    
\def\blue{\color{blue}}

\begin{document}
\numberwithin{equation}{section}
\title[Weight distribution of cyclic codes and curves]{Weight distribution of cyclic codes defined by \\ quadratic forms and related curves}
\author
{Ricardo A.\@ Podest\'a, Denis E.\@ Videla}
\thanks{Partially supported by CONICET, FONCYT and SECyT-UNC}
\email{podesta@famaf.unc.edu.ar}
\email{devidela@famaf.unc.edu.ar}
\address{Ricardo A.\@ Podest\'a -- CIEM, Universidad Nacional de C\'ordoba, CONICET, FAMAF. \linebreak 
	Av.\@ Medina Allende 2144, Ciudad Universitaria, (X5000HUA) C\'ordoba, República Argentina.} 
\address{Denis E.\@ Videla      -- CIEM, Universidad Nacional de C\'ordoba, CONICET, FAMAF. \linebreak 
	Av.\@ Medina Allende 2144, Ciudad Universitaria, (X5000HUA) C\'ordoba, República Argentina.} 

\begin{abstract}
We consider cyclic codes $\CC_\LL$ associated to 
quadratic trace forms in $m$ variables $Q_R(x) = \tr_{q^m/q}(xR(x))$ determined by a family $\LL$ of $q$-linearized polynomials $R$ over $\ff_{q^m}$, and three related codes $\CC_{\LL,0}$, $\CC_{\LL,1}$ and $\CC_{\LL,2}$.
We describe the spectra for all these codes when $\LL$ is an even rank family, 
in terms of the distribution of ranks of the forms $Q_R$ in the family $\LL$, 
and we also compute the complete weight enumerator for $\CC_\LL$. 
In particular, considering the family $\LL = \langle x^{q^\ell} \rangle$, with $\ell$ fixed in $\mathbb{N}$, we give the weight distribution of four parametrized families of cyclic codes $\CC_\ell$, $\CC_{\ell,0}$, $\CC_{\ell,1}$ and $\CC_{\ell,2}$ over 
$\mathbb{F}_q$ with zeros $\{ \alpha^{-(q^\ell+1)} \}$, $\{ 1,\, \alpha^{-(q^\ell+1)} \}$, $\{ \alpha^{-1},\,\alpha^{-(q^\ell+1)} \}$ and  $\{ 1,\,\alpha^{-1},\,\alpha^{-(q^\ell+1)}\}$ respectively, where $q = p^s$ with $p$ prime, $\alpha$ is a generator of 
$\mathbb{F}_{q^m}^*$ and $m/(m,\ell)$ is even. 
Finally, we give simple necessary and sufficient conditions for Artin-Schreier curves $y^p-y = xR(x) + \beta x$, $p$ prime, associated to polynomials $R \in \LL$ to be optimal.
We then obtain several maximal and minimal such curves in the case $\LL = \langle x^{p^\ell}\rangle$ and 
$\LL = \langle x^{p^\ell}, x^{p^{3\ell}} \rangle$.  
\end{abstract}

\maketitle

\section{Introduction}
Let $q=p^s$, with $p$ prime. A linear code $\CC$ of length $n$ over $\mathbb{F}_{q}$ is a subspace of $\mathbb{F}_{q}^n$ of 
dimension $k$. If $\CC$ has minimal distance $d = \min \{d(c,c') : c,c' \in \CC, c\ne c'\}$, where $d(\cdot , \cdot)$ is the Hamming distance in $\ff_q^n$, then $\CC$ is called an $[n,k,d]$-code. One of the most important families of codes are the cyclic ones. A code is cyclic if given a codeword $c=(c_1,\ldots,c_n) \in \CC$ the cyclic shift $s(c)=(c_n,c_1,\ldots,c_{n-1})$ is also in $\CC$.   
The weight of $c \in \CC$ is $w(c) = \# \{0\le i\le n : c_i\ne0 \}$; that is, the number of non-zero coordinates of $c$. 
For $0 \le i \le n$ the numbers $A_i=\#\{ c\in \CC : w(c)=i\}$ are called the frequencies and the sequence
$Spec(\CC) = (A_0,A_1,\ldots,A_n)$ is called \textit{the weight distribution} or \textit{the spectrum} of $\CC$.
A good reference for general coding theory is the book \cite{HP}.

Fix $\alpha$ a generator of $\mathbb{F}_{q^m}^*$. Consider $h(x) = h_1(x) \cdots h_t(x) \in \mathbb{F}_{q}[x]$ where $h_{j}(x)$ are different irreducible polynomials over $\mathbb{F}_{q}$. For each $j=1,\ldots, t$, let $g_j = \alpha^{-s_j}$ be a root of $h_j(x)$, 
$n_j$ be the order of $g_j$ and $m_j$ be the minimum positive integer such that $q^{m_j} \equiv 1 \pmod{n_j}$. 
Then, $\deg(h_{j}(x))=m_j$ for all $j$. 
Put $n = \tfrac{q^m-1}{\delta}$ where $\delta = \gcd(q^m-1, s_1,\ldots, s_t)$. 
Then, by Delsarte's Theorem of trace and duals (\cite{D}), the $q$-ary code 
$\CC=\{ c(a_1,\ldots,a_t) : a_j \in \mathbb{F}_{q^{m_j}} \}$ 
with 
\begin{equation} \label{delsarte} 
c(a_1,\ldots,a_{t}) = \Big( \sum_{j=1}^{t}{\tr_{q^{m_j}/q}(a_j)}, \sum_{j=1}^{t}{\tr_{q^{m_j}/q}(a_{j} g_{j})}, 
\ldots, \sum_{i=1}^{t}{\tr_{q^{m_j}/q}(a_{j}g_{j}^{n-1})} \Big),
\end{equation}
where $\tr_{q^{m_j}/q}$ is the trace function from $\mathbb{F}_{q^{m_j}}$ to $\mathbb{F}_q$, 
is an $[n,k]$-cyclic code with check polynomial $h(x)$ and dimension $k = m_1+\cdots+m_t$.

The computation of the spectra of cyclic codes is in general a difficult task. 
The recent survey \cite{DLY} of Dinh, Li and Yue shows the progress made on this problem in the last 20 years using different techniques: exponential sums, special nonlinear functions over finite fields, quadratic forms, Hermitian forms graphs, Cayley graphs, Gauss and Kloosterman sums. 
In \cite{FL1}, Feng and Luo computed the weight distribution of the cyclic code of length $n=p^m-1$ with zeros $\{ \alpha ^{-1}, 
\alpha^{-(p^\ell+1)} \}$, where $\alpha$ is a generator of $\mathbb{F}_{p^m}^*$, $\ell\geq 0$ and $m/(m,\ell)$ odd, by using a perfect nonlinear function. In another work (\cite{FL}), they used quadratic forms to calculate the weight distribution of the cyclic codes with zeros $\{ \alpha^{-2}, \alpha^{-(p^\ell+1)} \}$ and $\{ \alpha^{-1}, \alpha^{-2}, \alpha^{-(p^\ell+1)} \}$, respectively, when $p$ is an odd prime and $(m,\ell)=1$. These methods were used by other authors to calculate the spectra 
of other cyclic codes over 
$\mathbb{F}_p$ when $p$ is an odd prime. All these results are summarized in Theorem~2.4 in \cite{DLY}.

\medskip

In this paper, we will explicitly compute the weight distributions of some general families of cyclic codes over $\mathbb{F}_q$. In particular, we will compute the spectra of cyclic codes with zeros 
$\{\alpha^{-(q^\ell+1)}\}$, $\{1, \alpha^{-(q^\ell+1)}\}$, $\{ \alpha^{-1}, \alpha^{-(q^\ell+1)}\}$ and 
$\{1,\alpha^{-1},\alpha^{-(q^\ell+1)}\}$ in all characteristics, where $\alpha$ is a  generator of $\mathbb{F}_{q^m}^{*}$ and 
$m/(m,\ell)$ is even (i.e.\@ new cases not covered in \cite{FL1} and more general ones), by using quadratic forms and related exponential sums. 

We now give a brief summary of the results in the paper.
In Section 2 we recall quadratic forms $Q$ in $m$ variables over finite fields and their absolute invariants: the rank and the type. 
We define certain exponential sums $S_{Q,b}(\beta)$ and compute their values and distributions (Lemma~\ref{SQ}). 
We then consider the particular quadratic form $Q_{\gamma, \ell}(x) = \tr_{q^m/q}(\gamma x^{q^\ell+1})$, with $\gamma \in \ff_q$, $\ell \in \mathbb{N}$. We recall the distribution of rank and types given by Klapper in \cite{K1} and \cite{K2}.  
These facts will be later used (Sections 3--5) to compute the spectra of some families of cyclic codes.

In the next section, we consider cyclic codes defined by general quadratics forms determined by $q$-linearized polynomials and compute their spectra in some cases. More precisely, we consider $\LL = \langle x^{q^{\ell_1}}, x^{q^{\ell_2}}, \ldots, x^{q^{\ell_s}} \rangle \subset \ff_{q^m}[x]$, the associated code 
$$\CC_\LL = \{ (\tr_{q^m/q}(xR(x)))_{x\in \ff_{q^m}^*} : R\in \LL \}$$  
and three related codes $\CC_{\LL,0}$, $\CC_{\LL,1}$ and $\CC_{\LL,2}$ (see \eqref{codigoCL0}). 
If $\LL$ is an even rank family (see Definition~\ref{L even rank}) we give the weight distributions of 
$\CC_{\LL}$, $\CC_{\LL,0}$, $\CC_{\LL,1}$ and $\CC_{\LL,2}$ (see Theorems~\ref{Spec C_L} and \ref{Spec C_L,1} and Tables~1 to 4).
In Proposition~\ref{cweCL}, we also give the complete weight enumerator of $\CC_{\LL}$.

In the next sections we consider two particular even rank families: $\LL= \langle x^{q^\ell} \rangle$ and 
$\LL = \langle x^{q^\ell}, x^{q^{3\ell}} \rangle$, with $\ell \in \mathbb{N}$.
In Section 4, we compute the spectrum of the code $\CC_\ell$ defined by the family of quadratic forms
$Q_{\gamma,\ell} = \tr_{q^m/q}(\gamma x^{q^\ell+1})$, $\gamma \in \ff_{q^m}$, 
and the spectra of the related codes $\CC_{\ell,0}$, $\CC_{\ell,1}$ and $\CC_{\ell,2}$ (see Theorems \ref{DistCl} and \ref{Dist Cl1} and Tables 5--8). As a consequence, $\CC_\ell$ turns out to be a 2-weight code. The complete weight enumerator of $\CC_\ell$ is given in Corollary \ref{cwe Cl}. In Section 5 we obtain similar results for the codes $\CC_{\ell,3\ell}$, $\CC_{\ell,3\ell,0}$, 
$\CC_{\ell,3\ell, 1}$ and $\CC_{\ell,3\ell,2}$ (see Theorem \ref{spec cl3l} and Tables 9 and 10).

In the last section, we consider Artin-Schreier curves of the form 
$$ C_{R,\beta} \, : \, y^p-y = xR(x) + \beta x$$ 
where $p$ is prime, $\beta \in \ff_{p^m}$ and $R$ is a $p$-linearized polynomial over $\ff_{p^m}$. 
In Proposition 6.1 we give simple necessary and sufficient conditions for these curves to be optimal, that is curves attaining the equality in the Hasse-Weil bound, 
in terms of the degree of $R$ and the rank $r$ of the associated quadratic form $Q_R(x) = \tr_{p^m/p}(xR(x))$. 
We then show in Theorem 6.3 that there are several maximal and minimal curves in the family 
$$y^p-y = \gamma x^{p^\ell+1}+\beta x, \qquad \gamma \in \mathbb{F}_{p^m}^*, \quad \beta\in \mathbb{F}_{p^m}.$$ 
In the binary case $p=2$, Van der Geer and Van der Vlugt have found the same curves for $\ell=1$ and $\beta=0$ (see \cite{VGVV1}).
Thus, we extend their result for any $p$, $\ell$ and $\beta$. 
We also show the existence of optimal curves of the form
$$x^p-y = \gamma_1 x^{p^\ell+1} + \gamma_3 x^{p^{3\ell}+1} + \beta x$$
with $\gamma_1, \gamma_3 \in \mathbb{F}_{p^m}^*$, $\beta \in \mathbb{F}_{p^m}$.

\section{Quadratic forms over finite fields and exponential sums}
A quadratic form in $\ff_{q^m}$ is an homogeneous polynomial $q(x)$ in $\ff_{q^m}[x]$ of degree 2. 
We want to consider more general functions. 
Any function 
$$Q : \mathbb{F}_{q^m} \rightarrow \mathbb{F}_{q}$$ 
can be identified with a polynomial of $m$ variables over $\mathbb{F}_{q}$ via an isomorphism $\mathbb{F}_{q^m}\simeq\mathbb{F}_{q}^m$ of $\mathbb{F}_{q}$-vector spaces. 
Such $Q$ is said to be a \textit{quadratic form} if the corresponding polynomial is homogeneous of degree $2$. 
The \textit{rank} of $Q$, is the minimum number $r$ of variables needed to represent $Q$ as a polynomial in several variables.
Alternatively, the rank of $Q$ can be computed as the codimension of the $\mathbb{F}_{q}$-vector space 
$V = \{ y\in \mathbb{F}_{q^m}: Q(y)=0, Q(x+y)=Q(x),  \forall x\in \mathbb{F}_{q^m} \}$.
That is $|V|=q^{m-r}$. 
Two quadratic forms $Q_{1},Q_{2}$ are \textit{equivalent} if there is an invertible $\mathbb{F}_{q}$-linear function $S:\mathbb{F}_{q^m}\rightarrow\mathbb{F}_{q^m}$ such that $Q_1(x) =  Q_{2}(S(x))$.

Fix $Q$ a quadratic form from $\mathbb{F}_{q^m}$ to $\mathbb{F}_{q}$. It will be convenient to consider, for each 
$\beta\in\mathbb{F}_{q^m}$ and $\xi \in \mathbb{F}_{q}$, the number  
\begin{equation} \label{NQ}
N_{Q,\beta}(\xi) = \# \{ x\in \ff_{q^m} : Q(x) + \tr_{q^m/q}(\beta x) =\xi \}.
\end{equation}
We will abbreviate 
$N_{Q}(\xi)=N_{Q,0}(\xi)$, $N_{Q,\beta}=N_{Q,\beta}(0)$ and 
$N_Q = N_Q(0)=\#\ker Q$.
It is a classic result that quadratic forms over finite fields are classified in three different equivalent classes. This classification depends on the parity of the characteristic (see for instance \cite{LN}). 
But in both characteristics (even or odd), there are 2 classes with even rank (usually called type 1 and 3) and one of odd rank. 
For even rank, we will use the notation  
\begin{equation} \label{EQ}
\varepsilon_Q = \begin{cases} 
+1 & \qquad \text{if $Q$ is of type 1}, 
\\[1mm]-1 & \qquad \text{if $Q$ is of type 3},
\end{cases}
\end{equation}
and call this sign the \textit{type} of $Q$.
The number $N_Q(\xi)$ does not depend on the characteristic and it is given by 
\begin{equation} \label{NQ+-} 
N_Q(\xi) = 
q^{m-1} + \varepsilon_Q \, \nu(\xi) \, q^{m-\frac{r}{2}-1} 
\end{equation}
where $\nu(0)=q-1$ and $\nu(z)=1$ if $z\in \mathbb{F}_{q}^*$ (see \cite{LN}). From the works \cite{K1}, \cite{K2} of Klapper we also know the distribution of these numbers $N_{Q,\beta}(\xi)$, which are given as follows.

\begin{lem}\label{Result Kl}
Let $Q$ be a quadratic form of $m$ variables over $\mathbb{F}_{q}$ of even rank $r$. Then, for all $\xi \in \mathbb{F}_{q}$, there are $q^m-q^r$ elements $\beta \in \ff_{q^m}$ such that $N_{Q,\beta}(\xi)=q^{m-1}$ and $q^{r-1} + \varepsilon_Q \, \nu(c) \, q^{\frac{r}{2}-1}$ elements $\beta \in \ff_{q^m}$ such that $N_{Q,\beta}(\xi) = q^{m-1} + \varepsilon_Q \, \nu(\xi+c) \, q^{m-\frac{r}{2}-1}$, where $c$ runs on $\mathbb{F}_{q}$.
\end{lem}

Given a quadratic form $Q$, we can define the exponential sums
\begin{equation}\label{SQb}
S_{Q,b}(\beta) = \sum_{a\in\mathbb{F}_{q}^*} \sum_{x\in\mathbb{F}_{q^m}} \zeta_{p}^{\tr_{q/p} 
\left(a(Q(x)+\tr_{q^{m}/q}(\beta x)+b) \right)}
\end{equation}
where $\zeta_p = e^{\frac{2\pi i}{p}}$ and put $S_{Q}(\beta) = S_{Q,0}(\beta)$. 
We now give the values of $S_{Q,b}(\beta)$ and their distributions.

\begin{lem} \label{SQ}
Let $Q(x)$ be a quadratic form over $\mathbb{F}_{q}$ of even rank $r$. Then, 
\begin{align*} 
S_{Q}(\beta) & = \begin{cases}
	0      														& \qquad q^m-q^r \text{ times}, \\[1mm]
	\varepsilon(q-1)q^{m-\frac{r}{2}} & \qquad q^{r-1} + \varepsilon(q-1)q^{\frac{r}{2}-1}  \text{ times}, \\[1mm]
	-\varepsilon q^{m-\frac{r}{2}} 		& \qquad (q^{r-1}-\varepsilon q^{\frac{r}{2}-1})(q-1) \text{ times}.
	\end{cases} \\[1.25mm]
S_{Q,b}(\beta) & =
	\begin{cases}
	0       													& \qquad q^m-q^r  \text{ times}, \\[1mm]
	\varepsilon(q-1)q^{m-\frac{r}{2}} & \qquad q^{r-1}-\varepsilon q^{\frac{r}{2}-1}  \text{ times}, \\[1mm]
	-\varepsilon q^{m-\frac{r}{2}} 		& \qquad q^{r}-q^{r-1}+\varepsilon q^{\frac{r}{2}-1}  \text{ times}.
	\end{cases} 
\end{align*}
\end{lem}

\begin{proof}
Notice that
$$S_{Q}(\beta) = \sum_{a\in\mathbb{F}_{q}^*}\sum_{x\in\mathbb{F}_{q^m}}\zeta_{p}^{\tr_{q/p}(a(Q(x)+\tr_{q^{m}/q}(\beta x)))}
=  \sum_{x\in\mathbb{F}_{q^m}}\sum_{a\in\mathbb{F}_{q}}\zeta_{p}^{\tr_{q/p}(a(Q(x)+\tr_{q^{m}/q}(\beta x)))} -q^m.$$
Hence, $S_{Q}(\beta) = q N_{Q,\beta}(0)-q^m$.
Similarly, we get $S_{Q,b}(\beta) = q N_{Q,\beta}(-b)-q^m$. The result now follows from Lemma \ref{Result Kl}. 
\end{proof}

\subsubsection*{The quadratic form $\tr_{q^m/q}(\gamma x^{q^\ell+1})$}
A whole family of quadratic forms over $\ff_q$ in $m$ variables are given by  
\begin{equation} \label{QRx}
Q_R(x) = \tr_{q^m/q}(xR(x))
\end{equation} 
where $R(x)$ is a $q$-linearized polynomial over $\ff_q$. 
We are interested in the simplest case, when $R(x)$ is the monomial $R_{\gamma, \ell}(x) = \gamma x^{q^\ell}$ with 
$\ell \in \mathbb{N}$, $\gamma \in \ff_{q^m}^*$, i.e.\@
\begin{equation} \label{Qell}
Q_{\gamma, \ell}(x) = \tr_{q^m/q}(\gamma x^{q^\ell+1}). 
\end{equation}
The next theorems, due to Klapper, give the distribution of ranks and types of the family of quadratics forms 
$\{ Q_{\gamma, \ell}(x) = \tr_{q^m/q}(\gamma x^{q^\ell+1}) : \gamma \in \mathbb{F}_{q^m}, \ell \in \mathbb{N} \}$.

\smallskip 
For integers $m, \ell$ we will use the following notations 
$$m_{\ell} = \tfrac{m}{(m,\ell)} \qquad \text{and} \qquad \varepsilon_{\ell}=(-1)^{\frac 12 m_{\ell}}$$
and denote the set of $(q^\ell+1)$-th powers in $\ff_{q^m}$ by 
\begin{equation} \label{Sqml}
S_{q,m}(\ell) = \{ x^{q^\ell+1} : x \in \ff_{q^m}^* \}.
\end{equation}

\begin{thm}[even characteristic (\cite{K1})] \label{Thmpar}
Let $q$ be a power of $2$ and $m,\ell \in \mathbb{N}$ such that $m_\ell$ is even. Then $Q_{\gamma,\ell}$ is of even rank and we have:
\begin{enumerate}[(a)]
\item If $\varepsilon_\ell = \pm 1$ and $\gamma \in S_{q,m}(\ell)$ then $Q_{\gamma, \ell}$ is of type $\mp 1$ and has rank
 $m-2{(m,\ell)}$.
		
\item If $\varepsilon_\ell = \pm 1$ and $\gamma \notin S_{q,m}(\ell)$ then $Q_{\gamma,\ell}$ is of type $\pm 1$ and has rank $m$.
\end{enumerate}
\end{thm}

For $q$ odd, consider the following sets of integers 
\begin{equation} \label{Xqml}
X_{q,m}(\ell) = \{ 0\le t\le N :  t \equiv 0 \, (L)\} 
\quad \text{and} \quad  
Y_{q,m}(\ell) = \{0\le t\le N : t \equiv \tfrac L2 \, (L) \}
\end{equation}
where $N=q^m-1$ and $L=q^{(m,\ell)}+1$.

\begin{thm}[odd characteristic (\cite{K2})] \label{Thmimpar}
Let $q$ be a power of an odd prime $p$ and let $m, \ell$ be non negative integers. Put $\gamma = \alpha^{t}$ with $\alpha$ a primitive element in $\ff_{q^m}$. Then, we have:
	\begin{enumerate}[(a)]
		\item If $\varepsilon_\ell = 1$ and $t \in X_{q,m}(\ell)$ then $Q_{\gamma,\ell}$ is of type $-1$ and has rank $m-2(m,\ell)$.
		
		\item If $\varepsilon_\ell = 1$ and $t\notin X_{q,m}(\ell)$ then $Q_{\gamma,\ell}$ is of type $1$ and has rank $m$.
		
		\item If $m_\ell$ is even, $\varepsilon_\ell = - 1$ and $t \in Y_{q,m}(\ell)$ 
		then $Q_{\gamma,\ell}$ is of type $1$ and has rank $m-2{(m,\ell)}$.
		
		\item If $m_\ell$ is even, $\varepsilon_\ell = -1$ and $t\notin Y_{q,m}(\ell)$ 
		then $Q_{\gamma,\ell}$ is of type $-1$ and has rank $m$.
\end{enumerate}
\end{thm}

We will need the following result whose proof is elementary.
\begin{lem} \label{mcd}
Let $q$ be a prime power and $m, \ell$ integers. If $m_\ell$ is even then $(q^m-1, q^\ell+1) = q^{(m,\ell)}+1$.
\end{lem}

\begin{lem}\label{LemaEmes}
Let  $M = \# S_{q,m}(\ell)$, $M_1 =\#X_{q,m}(\ell)$ and $M_2 =\#Y_{q,m}(\ell)$ and put 
$M' = q^m-1-M$, $M_1' =q^m-1-M_1$ and $M_2' =q^m-1-M_2$.
If $m_\ell$ is even then 
\begin{equation} \label{fla M} 
M = M_1 = M_2 = \frac{q^m-1}{q^{(m,\ell)}+1} 
\qquad \text{and} \qquad 
 M' = M'_1 = M'_2 = q^{(m,\ell)}\frac{q^m-1}{q^{(m,\ell)}+1}. 
\end{equation}
\end{lem}

\begin{proof}
	Let $\alpha$ be a primitive element of $\mathbb{F}_{q^m}$, then $S_{q,m}(\ell) = \langle \alpha^{q^\ell+1} \rangle$, this implies that 
	$M = \frac{q^m-1}{(q^m-1,q^{\ell}+1)} = \frac{q^m-1}{q^{(m,\ell)}+1}$, by Lemma \ref{mcd}. 
	On the other hand, if $k$, $N$, $s_1$, $s_2$ are non-negative integers with $k\mid N$ and 
$0 \le s_1, s_2 \le k-1$, then 
	$$\#\{ i\in\{1,\ldots,N\} : i \equiv s_{1} \text{ mod } k \} = \#\{ i\in\{1,\ldots,N\}: i\equiv s_{2} \text{ mod } k \}  = 
	\tfrac Nk.$$ 
	Therefore, if $q^{(m,\ell)}+1\mid q^{m}-1$, we have that 
	$M_1 = M_2 = \frac{q^m-1}{q^{(m,\ell)}+1}$. 
    Clearly, we obtain that $M'=M'_1=M'_2=q^{(m,\ell)} M$ as we wanted.  
\end{proof}

\section{Weight distribution of cyclic codes defined by trace forms}
Let $\mathcal{L} \subset \mathbb{F}_{q^m}[x]$ be a finite dimensional $\mathbb{F}_{q^m}$-subspace containing  
$q$-linearized polynomials only, i.e. 
$$\mathcal{L} = \langle x^{q^{\ell_1}}, x^{q^{\ell_2}}, \ldots, x^{q^{\ell_s}} \rangle \subset \mathbb{F}_{q^m}[x]$$ 
for some non-negative integers $\ell_1, \ldots,\ell_s$ with 
$\ell_i\ne\ell_j$ for $i\ne j$.
Define the $q$-ary code 
\begin{equation}\label{codigoCL}
\CC_{\mathcal{L}} = \Big \{ c_R= \big( \tr_{q^m/q}(xR(x)) \big)_{x\in\mathbb{F}_{q^m}^*}: R\in\mathcal{L} \Big \} \subset 
\mathbb{F}_{q}^n
\end{equation}
with length $n=q^m-1$ and the related codes
\begin{equation}\label{codigoCL0}
\begin{split}
& \CC_{\mathcal{L},0} = \Big\{ c_{R,b} = \big( \tr_{q^m/q} \big( xR(x) ) +b \big)_{x\in\mathbb{F}_{q^m}^*}: R\in\mathcal{L},\text{ } b\in\mathbb{F}_{q} 
\Big\}, \\[1mm]
& \CC_{\mathcal{L},1} = \Big\{ c_{R}(\beta) = \big( \tr_{q^m/q}(xR(x)+\beta x) \big)_{x\in\mathbb{F}_{q^m}^*}: R\in\mathcal{L}, \ \beta\in\mathbb{F}_{q^m} 
\Big\}, \\[1mm]
& \CC_{\mathcal{L},2} = \Big\{ c_{R,b}(\beta) = \big( \tr_{q^m/q}(xR(x)+\beta x)+b \big)_{x\in\mathbb{F}_{q^m}^*}: R\in\mathcal{L}, \ \beta\in\mathbb{F}_{q^m}, 
\ b \in\mathbb{F}_{q} \Big\} .
\end{split}
\end{equation}
Notice that $c_{R,b}=c_R+b$ and $c_{R,b}(\beta) = c_{R}(\beta)+b$; moreover, we have that $c_{R,0}=c_R(0)=c_R$, $c_{R,b}(0)=c_{R,b}$ and 
$c_{R,0}(\beta)=c_R(\beta)$. Then, we have  
$$\CC_{\mathcal{L}} \subset \CC_{\mathcal{L},0}, \, \CC_{\mathcal{L},1} \subset \CC_{\mathcal{L},2}.$$

All of these codes are cyclic since one can check that their words have the form \eqref{delsarte}. In our case, this can be seen  directly. 
If $\alpha$ is a primitive element of $\ff_{q^m}$ then 
$$c_R=(\tr_{q^m/q} (xR(x)))_{x\in \ff_{q^m}^*} = (\tr_{q^m/q} ( \alpha^i R(\alpha^i)))_{i=0}^{q^m-2}.$$ 
The cyclic shift $s(c_R)=(\tr_{q^m/q} ( \alpha^i R(\alpha^i)))_{i=0}^{q^m-2} = c_S$ with $S(x) = \alpha^{-1} R(\alpha^{-1}x) \in \mathcal{L}$, hence $s(c_R)$ is in $\CC_\mathcal{L}$ and the code is cyclic. Similarly for the other codes.

\begin{defi} \label{L even rank}
A family $\mathcal{L} \subset\mathbb{F}_{q^m}[x]$ of $q$-linearized polynomials 
has the \textit{even rank property} or is an  \textit{even rank family} if 
the quadratic form $Q_{R}(x)=\tr_{q^m/q}(xR(x))$ has even rank for any $R\in\mathcal{L}$. 
\end{defi}

Let $\mathcal{L}$ be an even rank family of $q$-linearized polynomials. Then,
$Q_{R}(x) =  \tr_{q/p}(xR(x))$ has constant type in the family; that is $Q_R(x)$ is of type 1 or of type $-$1 for every 
$R\in\mathcal{L}$.
Therefore, given $r$ a non-negative integer, we can define 
\begin{equation}\label{DefKri}
\begin{split} 
K_{r} & =\{R\in\mathcal{L}: Q_{R} \text{ has rank } r \}, \\[1.5mm]
K_{r,1} & =\{R\in\mathcal{L} \smallsetminus \{0\}: Q_{R} \text{ has rank } r \text{ of type 1} \}, \\[1.5mm]
K_{r,2} & =\{R\in\mathcal{L} \smallsetminus \{0\}: Q_{R} \text{ has rank } r \text{ of type 3} \}. 
\end{split}
\end{equation}
We have $K_{0}=\{0\}$ and $K_r = K_{r,1} \sqcup K_{r,2}$ for $r>0$, and we denote their cardinalities by
\begin{equation} \label{M1r}  
M_{r,1} =\#K_{r,1}, \qquad  M_{r,2} =\#K_{r,2}, \qquad M_r = \#K_r. 
\end{equation}   
Note that $M_0=1$ and $M_r = M_{r,1}+M_{r,2}$ for $r>0$. 
Finally, we denote the set of ranks in $\mathcal{L}$ by
\begin{equation} \label{LR} 
R_{\mathcal{L}} =\{r\in\mathbb{Z}_{\geq 0}: \text{exists } R\in\mathcal{L} \text{ with } Q_{R} \text{ of rank } r \}.
\end{equation}

For any positive integer $r$, we define the set  
\begin{equation}\label{[r]_q}
[r]_q := \{ q^{\ell}+1 : 0< \ell < r \} = \{q+1, q^2+1, \ldots,q^{r-1}+1\}. 
\end{equation}
We now restate Lemma 2.1 in \cite{ZZDX2} in more generality and give a proof for completness. We will need it to calculate the dimensions of the four families of codes considered in this section.
\begin{lem}\label{dim}
Let $m$ be a positive even integer and let $M=\{1\}\cup [\frac m2]_q$. If $\alpha$ is a primitive element of $\ff_{q^m}$ then we have: 
\begin{enumerate}[$(a)$]
	\item $\alpha^{-u}$ and $\alpha^{-v}$ are not conjugated for all distinct elements $u,v \in M$.
	\item The minimum $m_u \in \{1,\ldots,m\}$ such that $q^{m_u} u \equiv u \pmod{q^m-1}$ is $m$ for all $u\in M$. 
\end{enumerate}
\end{lem}
\begin{proof}
	For $(a)$ it is enough to show that $q^{s}(q^{\ell_1}+1)\not\equiv q^{\ell_2}+1\pmod{q^m-1}$ and $q^s\not\equiv q^{\ell}+1\pmod{q^m-1}$ for $1\le s\le m-1$ and $\ell,\ell_1 ,\ell_2 <\frac m2$ with $\ell_1 \neq \ell_2$. We will show the first statement. Suppose that there is some $s\in \{1,\ldots,m\}$ such that 
	$q^{s}(q^{\ell_1}+1)\equiv q^{\ell_2}+1\pmod{q^m-1}$. Then,
	$$q^{s+\ell_1}+q^s\equiv q^{\ell_2}+1\pmod{q^m-1}$$
	If $s+\ell_1<m$ then $q^{s+\ell_1}+q^s= q^{\ell_2}+1$ as integers. 
	The uniqueness of the $q$-ary expansion of integers implies that $s+\ell_1=\ell_2$ and $s=0$, which cannot happen. 
	Now, if $s+\ell_1>m$ then $s>\frac m2$ since by hypothesis $\ell_1 < \frac m2$, 
	and hence there exists a positive integer $t<\frac m2$ such that $s=m-t$. Notice that $\ell_1 >t$ and $0<\ell_1-t<\frac m2$, thus
	$q^{s+\ell_1}+q^s\equiv q^{\ell_1-t}+q^s\pmod{q^m-1}$
	and hence 
	$$q^{\ell_1-t}+q^s \equiv q^{\ell_2}+1 \pmod{q^m-1}$$
	Since all the powers are less than $m$, we obtain $q^{\ell_1-t}+q^s= q^{\ell_2}+1$ as integers, 
	by unicity of $q$-ary expansion we obtain that $s=\ell_2$ and $\ell_1 -t=0$ which cannot occur. 
	Therefore, $\alpha^{-u}$ and $\alpha^{-v}$ are not conjugated for all $u\neq v$ in $[\frac m2 ]_q$. 
	In a similar way, it can be shown that $\alpha^{-1}$ and $\alpha^{-u}$ are not conjugated for all $u\in [\frac  m2]_q$.
	
	The item $(b)$ can be proved by a similar argument as in $(a)$.
\end{proof}	

We are now in a position to give the weight distribution of the four codes considered. We give the spectra in two theorems.

\begin{thm} \label{Spec C_L}
Let $q$ be a prime power, $m$ a non-negative integer and $\mathcal{L}=\langle x^{q^{\ell_1}}, x^{q^{\ell_2}}, \ldots, x^{q^{\ell_s}} \rangle$ an ideal in $\mathbb{F}_{q^m}[x]$ such that $1 \le \ell_1 < \ell_2 < \cdots < \ell_s < \frac m2$. If $\mathcal{L}$ is an even rank family then the dimensions of the cyclic codes $\cod_{\mathcal{L}}$ and $\cod_{\mathcal{L},0}$ are $ms$ and $ms+1$ respectively and their spectra are given by Tables 1 and 2 below.
\end{thm}

\begin{table}[H]
	\centering
	\caption{Weight distribution of $\CC_{\mathcal{L}}$ ($r \in R_{\mathcal{L}}$, $i=1,2$).} 
		\begin{tabular}{|c|c|c|}
		\hline 
		weight & frequency \\ 
		\hline \hline
		$w_0= 0$ & $1$ \\
		\hline
		$w_{1,i} = q^m-q^{m-1}+(-1)^{i}(q-1)q^{m-\frac{r}{2}-1}$ & $M_{r,i}$  \\ 
		\hline  
	\end{tabular}
\end{table}

\begin{table}[H]
	\centering
	\caption{ Weight distribution of $\CC_{\mathcal{L},0}$ ($r \in R_{\mathcal{L}}$, $i=1,2$).}
	\begin{tabular}{|c|c|c|}
		\hline 
		weight & frequency \\ 
		\hline \hline
		$w_0=0$ & $1$ \\
		\hline
		$w_1=q^m-1$ & $q-1$ \\
		\hline  
		$w_{2,i} = q^m-q^{m-1}+(-1)^{i}(q-1)q^{m-\frac{r}{2}-1}$ & $M_{r,i}$  \\ 
		\hline
		$w_{3,i} = q^m-q^{m-1}+(-1)^{i+1} q^{m-\frac{r}{2}-1}-1$ & $M_{r,i}(q-1)$ \\
		\hline
	\end{tabular}
\end{table}

\begin{proof}
By definition $w(c_R)=\#\{x\in\ff_{q^m}^*: Q_{R}(x)\neq 0\}$, then 
$$ w(c_R)=q^m-1-\#\{x\in\ff_{q^m}^*: Q_{R}(x)=0\}. $$
Analogously, $ w(c_{R,b}) = q^m-1 - \#\{x\in\ff_{q^m}^*: Q_{R}(x)=-b\}.$ 
Then we have that $$w(c_{R,b})=\left\{
		\begin{array}{lll}
		q^m-N_{Q_R}(0)    & &  \text{if } b=0 , \\[1.2mm]
		q^m-N_{Q_R}(-b)-1 & & \text{if } b\neq0.
		\end{array}
		\right.$$
 If $Q_{R}$ has rank $r$ and type $\varepsilon_{R}$ then 
$$w(c_{R,b})=\left\{
\begin{array}{lll}
q^m-q^{m-1} - \varepsilon_{R} \, (q-1) \, q^{m-\frac{r}{2}-1}    & &  \text{if } b=0, \\[1.5mm]
q^m-q^{m-1} + \varepsilon_{R} \, q^{m-\frac{r}{2}-1}-1 & & \text{if } b\neq0.
\end{array}
\right.$$
From these facts, using the numbers $M_r, M_{r_i}$ and the set $R_{\mathcal{L}}$, we obtain the weights and frequencies given in 
Tables 1 and 2, and the result thus follows.

Let us consider the polynomials $h(x)=\prod_{j=1}^s h_{\ell_j}(x), $
where $h_{\ell_j}(x)$ are the minimal polynomials of $\alpha^{-(q^{\ell_{j}}+1)}$ over $\ff_q$ with $\alpha$ a primitive element of $\ff_{q^m}$ and $j=1,\ldots,s$.  By Delsarte's Theorem, if $n=\frac{q^m-1}{\delta}$ with $\delta=\gcd(q^m-1,q^{\ell_1}+1,\ldots,q^{\ell_s}+1)$ then $h(x)$ is the check polynomial of the cyclic code 
\begin{equation}\label{def C^*}
\CC^{*}_{\mathcal{L}}=\{ c(a_1,\ldots,a_s)= \big( \sum_{j=1}^s \tr _{q^m/q}(a_j  g_{j}^i) \big)_{i=1}^n : a_j\in \ff_{q^m} \} \subset \ff_q^n
\end{equation}
where $g_{j}=\alpha^{q^{\ell_j}+1}$ for $j=1,\ldots,s$. 
Since the dimension of a cyclic code is given by the degree of its check polynomial, we have  
$$\dim \CC^{*}_{\mathcal{L}}=\deg h(x)$$ 
It is known, by general theory of finite fields, that the degree of the minimal polynomial over $\ff_q$ of an element $u\in \ff_q$ 
is given by the size of its cyclotomic coset, and this size coincides with the minimum $1\le m_{u}\le m$ such that $q^{m_u}u\equiv u\pmod{q^m-1}$. 
By Lemma \ref{dim}, all of the elements in $\mathcal{L}$ are not conjugated to each other and $\deg h_{\ell_j}(x)=m$ for $j=1,\ldots,s$. Hence $\deg h(x)=sm$ and thus $\dim \CC^{*}_{\mathcal{L}}=sm$.

On the other hand, if $R(x)=\sum_{j=1}^s a_j x^{q^{\ell_j}}\in \mathcal{L}$, by linearity of the trace function we have that
$$c_R = \big(\tr_{q^m/q}(xR(x))\big)_{x\in \ff_{q^m}^*} = \big( \sum_{j=1}^s \tr_{q^m/q}(a_j \alpha^{i(q^{\ell_j}+1)})\big)_{i=1}^{q^m-1} = \big( \sum_{j=1}^s \tr_{q^m/q}(a_j g_{j}^i)\big)_{i=1}^{q^m-1}.$$
Notice that if $n=\frac{q^m-1}{\delta}$  as before, by modularity 
we get
\begin{equation}\label{cycle}
\big( \sum_{j=1}^s \tr_{q^m/q}(a_j g_{j}^i)) \big)_{i=1}^{n} = \big( \sum_{j=1}^s \tr_{q^m/q}(a_j g_{j}^i)) \big)_{i=(t-1)n+1}^{tn} 
\end{equation}
for every $1\le t \le \delta$.
Thus, denoting $\mathbf{c}=c(a_1,\ldots,a_s)\in \CC_{\mathcal{L}}$, 
by \eqref{cycle} we have that 
$c_R = \big(\mathbf{c} \,| \cdots |\, \mathbf{c} \big)$ 
$\delta$-times.
Hence all the words in $\CC_{\mathcal{L}}$ are obtained by $\delta$-concatenation of the words of the cyclic code $\CC^*_{\mathcal{L}}$, this implies that the dimension of these codes are the same. Therefore $\dim \CC_{\mathcal{L}} = \dim \CC^*_{\mathcal{L}}=sm$. The same argument shows that $\dim \CC_{\mathcal{L},0}=sm+1$. 
\end{proof}

\begin{thm} \label{Spec C_L,1}
Let $q$ be a prime power, $m$ a non-negative integer and $\mathcal{L} = \langle x^{q^{\ell_1}}, x^{q^{\ell_2}}, \ldots, x^{q^{\ell_s}} \rangle$ an ideal in $\mathbb{F}_{q^m}[x]$ such that $1 \le \ell_1 < \ell_2 < \cdots < \ell_s < \frac m2$.
If $\mathcal{L}$ is an even rank family, then 
the dimensions of the cyclic codes $\cod_{\mathcal{L},1}$ and $\cod_{\mathcal{L},2}$ are $m(s+1)$ and $m(s+1)+1$ respectively and their spectra are given by Tables 3 and 4 below. 
\end{thm}

\begin{table}[ht!]
	\centering
	\caption{ Weight distribution of $\CC_{\mathcal{L},1}$  ($r \in R_{\mathcal{L}}$, $i=1,2$).} 
	\medskip
	\begin{tabular}{|c|c|c|}
		\hline 
		weight & frequency \\ 
		\hline \hline
		$w_0=0$ & $1$ \\
		\hline
		$w_{1,i} = q^m-q^{m-1}$ &  $\sum_{r\in R_{\mathcal{L}}}M_r(q^m-q^r) $ \\ 
		\hline 
		$w_{2,i} = q^m-q^{m-1}+(-1)^{i}(q-1)q^{m-\frac{r}{2}-1}$ & $M_{r,i}(q^{r-1}+(-1)^{i+1}(q-1)q^{\frac{r}{2}-1})$ \\ 
		\hline 
		$w_{3,i} = q^m-q^{m-1}+(-1)^{i+1}q^{m-\frac{r}{2}-1}$  & $M_{r,i}(q^{r-1}+(-1)^{i}q^{\frac{r}{2}-1})(q-1)$ \\ 
		\hline  
	\end{tabular}
\end{table}

\begin{table}[ht!]
	\centering
	\caption{ Weight distribution of $\CC_{\mathcal{L},2}$  ($r \in R_{\mathcal{L}}$, $i=1,2$).} 
	\medskip
	\begin{tabular}{|c|c|c|}
		\hline 
		weight & frequency \\ 
		\hline \hline
		$w_0=0$ & $1$ \\
		\hline
		$w_{1} = q^m-q^{m-1}$ &  $\sum_{r\in R_{\mathcal{L}}}M_r(q^m-q^r)$ \\ 
		\hline 
		$w_{2,i} = q^m-q^{m-1}+(-1)^{i}(q-1)q^{m-\frac{r}{2}-1}$ & $M_{r,i}(q^{r-1}+(-1)^{i+1}(q-1)q^{\frac{r}{2}-1})$  \\ 
		\hline 
		$w_{3,i} = q^m-q^{m-1}+(-1)^{i+1}q^{m-\frac{r}{2}-1}$  & $M_{r,i}(q^{r-1}+(-1)^{i}q^{\frac{r}{2}-1})(q-1)$ \\ 
		\hline  
		$w_{4} = q^m-1$ & $q-1$ \\
		\hline
		$w_{5} = q^m-q^{m-1}-1$ &  $(q-1)\sum_{r\in R_{\mathcal{L}}}M_r(q^m-q^r)$ \\ 
		\hline 
		$w_{6,i} = q^m-q^{m-1}-1+(-1)^{i}(q-1)q^{m-\frac{r}{2}-1}$ & $M_{r,i}(q^{r-1}+(-1)^{i}q^{\frac{r}{2}-1})(q-1)$  \\ 
		\hline 
		$w_{7,i} = q^m-q^{m-1}-1+(-1)^{i+1}q^{m-\frac{r}{2}-1}$  & $M_{r,i}(q^{r}-q^{r-1}+(-1)^{i+1}q^{\frac{r}{2}-1})(q-1)$ \\ 
		\hline  
	\end{tabular}
\end{table}

\begin{proof}
The dimensions of $\CC_{\mathcal{L},1}$ and $\CC_{\mathcal{L},2}$  can be obtained in the same way as in Theorem \ref{Spec C_L} using Lemma \ref{dim}.

Now, let $R\in \mathcal{L}$ and suppose the quadratic form $Q_{R}$ has rank $r$ and type $\varepsilon_{R}$.
Let's see the weights of the words of $\CC_{\mathcal{L},1}$. By the orthogonality property of the characters of $\ff_q$, we have that 
	\begin{eqnarray*}
		w(c_{R}(\beta)) &=& q^m-1 - \# \{x\in\mathbb{F}_{q^m}^*:Q_{R}(x)+\tr_{q^m/q}(\beta x)=0\} \\[1.5mm]
		&=& q^m-1 -\tfrac{1}{q} \sum_{a\in\mathbb{F}_{q}} \sum_{x\in\mathbb{F}_{q^m}^*} 
		\zeta_{p}^{\tr_{q/p}(a(Q_{R}(x)+\tr_{q^{m}/q}(\beta x)))} \\[1mm]
		&=& q^m-1 -\tfrac{1}{q} \big\{ \sum_{a\in\mathbb{F}_{q}}\sum_{x\in\mathbb{F}_{q^m}} 
		\zeta_{p}^{\tr_{q/p}(a(Q_{R}(x)+\tr_{q^{m}/q}(\beta x)))}-q \big\} \\[1mm]
		&=&q^m-1 -\tfrac{1}{q} \big\{ \sum_{a\in\mathbb{F}_{q}^*}\sum_{x\in\mathbb{F}_{q^m}} 
		\zeta_{p}^{\tr_{q/p}(a(Q_{R}(x)+\tr_{q^{m}/q}(\beta x)))}+q^m-q \big\} .
	\end{eqnarray*}
	Therefore  
	\begin{equation}
	w(c_{R}(\beta))=q^m-q^{m-1} -\tfrac{1}{q}S_{Q_{R}}(\beta),
	\end{equation}
	where $S_{Q_{R}}$ is the exponential sum \eqref{SQb}, with $b=0$.
	In the same way, when $b\neq0$, we get
	\begin{equation}
	w(c_{R,b}(\beta))= q^m-q^{m-1}-1 -\tfrac{1}{q}S_{Q_{R},b}(\beta).
	\end{equation} 
	Notice that if $R=0$, then $Q_{R}=0$ and, for all $\beta\neq 0$, we have 
	$w(c_{0}(\beta))=q^m-q^{m-1}$. 
	If $R$ and $\beta$ are zeros, then 
	$w(c_{0,b}(0))=q^m-1$ if $b\neq 0$. When $b=0$ we will denote $c_{R}(\beta)=c_{R,0}(\beta)$. 
	
	Now, let $K_{r,1}$ and $K_{r,2}$ be as in \eqref{DefKri}. Then, $\varepsilon_{R}=(-1)^{i+1}$ if $R\in K_{r,i}$, $i=1,2$. 
	By Lemma~\ref{SQ}, we have that
	{\footnotesize
		$$w(c_{R,b}(\beta)) = \begin{cases}
		0 &  \text{if } b=0, \, R=0, \\[1mm]
		q^m-q^{m-1} &  \text{if } b=0, \, R\in K_{r,i}, \text{ for } q^m-q^r \, \beta's, \\[1mm]
		q^m-q^{m-1}+(-1)^{i} (q-1)q^{m-\frac{r}{2}-1} & \text{if } b=0,\, R\in K_{r,i}, \text{ for } q^{r-1}+(-1)^{i+1}(q-1)q^{\frac{r}{2}-1} \, \beta's, \\[1mm]
		q^m-q^{m-1}+(-1)^{i+1}q^{m-\frac{r}{2}-1} &  \text{if } b=0,\, R\in K_{r,i}, \text{ for } (q^{r-1}+(-1)^{i}q^{\frac{r}{2}-1})(q-1) \, \beta's, \\[1mm]
		q^m-1 &  \text{if } b\neq 0, \, R=0, \, \beta=0, \\[1mm]
		q^m-q^{m-1}-1 &  \text{if } b\neq 0, \, R\in K_{r,i}, \text{ for } q^m-q^r \, \beta's, \\[1mm]
		q^m-q^{m-1}+(-1)^{i} (q-1)q^{m-\frac{r}{2}-1}-1 & \text{if } b\neq0,\, R\in K_{r,i}, \text{ for } q^{r-1}+(-1)^{i}q^{\frac{r}{2}-1} \, \beta's, \\[1mm]
		q^m-q^{m-1}+(-1)^{i+1}q^{m-\frac{r}{2}-1}-1 &  \text{if } b\neq0,\, R\in K_{r,i}, \text{ for } q^{r}-q^{r-1}+(-1)^{i+1}q^{\frac{r}{2}-1} \, \beta's, \\[1mm]
		\end{cases}		$$}
	with $i=1,2$. From this, the result readily follows.
\end{proof}

\begin{rem} \label{rho weight}
($i$) A code is \textit{$t$-divisible} if the weight of every codeword is divisible by $t$.
Note that from Tables 1--4, 
the code $\CC_\LL$ is $q^{m-\frac r2 -1} (q-1)$-divisible, the code $\CC_{\LL,0}$ es $(q-1)$-divisible if and only if $M_{r,2}=0$ and that $\CC_{\LL,1}$ is $q^{m-\frac r2 -1}$-divisible.  

($ii$) If one perform the sum of the frequencies in each of the Tables 1--4 one checks that the dimensions of the codes are the ones given in Theorems \ref{Spec C_L} and \ref{Spec C_L,1}.
\end{rem}

\subsubsection*{Complete weight enumerator}
Suppose that the elements of $\mathbb{F}_q$ are ordered by $\omega_0,\omega_1,\ldots,\omega_{q-1}$, where $\omega_0 = 0$. The composition of the vector $v = (v_0,v_1,\ldots,v_{n-1})\in \mathbb{F}_{q}^n $ is defined by 
$$\mathrm{comp}(v) = (t_0,t_1,\ldots,t_{q-1}),$$ 
where each $t_i =t_i(v)=\#\{0\le j\le n-1: v_j=\omega_i\}$. 
Clearly, we have that 
$\sum_{i=0}^{q-1}t_{i}=n$. 
Let $\CC$ be a linear code of length $n$ over $\mathbb{F}_{q}$ and let 
$$A(t_0,t_1,\ldots,t_{q-1})=\#\{c\in \CC: \mathrm{comp}(c) = (t_{0},t_{1},...,t_{q-1})\}.$$ 
The \textit{complete weight enumerator} of $\CC$ is the polynomial 
\begin{equation}
W_\CC(z_0,z_1,\ldots,z_{q-1}) = \sum_{(t_0,\ldots,t_{q-1})\in B_n} A(t_0,t_1,\ldots,t_{q-1}) \, 
z_{0}^{t_{0}}z_{1}^{t_{1}}\cdots z_{q-1}^{t_{q-1}},
\end{equation}
where $B_n=\{(t_0,\ldots,t_{q-1}) \,:\,  t_{i}\geq 0, \, t_0 + \cdots + t_{q-1} =n\}$. 
\begin{lem}
Let $\CC$ be a linear code of length $n$ over $\mathbb{F}_{q}$ such that $t_{i}(c)=t_{j}(c)$ for all $i,j>0$ and $c\in \CC$. 
Then, if $A_{\ell}=\#\{c\in \CC: w(c)=\ell\}$, we have that
$$W_\CC(z_0,z_1,\ldots,z_{q-1})=\sum_{\ell=0}^{n} A_{\ell} \, z_{0}^{n-\ell}z_{1}^{\frac{\ell}{q-1}}\cdots z_{q-1}^{\frac{\ell}{q-1}}.$$
\end{lem}

\begin{proof}
Let $c=(c_0,\ldots,c_{n-1})\in \CC$. Since $t_{i}(c)=t_{j}(c)$ for $i,j>0$ and $\sum_{i=0}^{q-1}t_{i}=n$, we have that 
	$t_{i} = \frac{n-t_0}{q-1}$. On the other hand, since 
	$w(c)=n-\#\{0\le j \le n-1: c_{j}=0\}=n-t_0$, 
	we have that $t_0=n-w(c)$, and thus $t_{1}=\frac{w(c)}{q-1}$ 
(note that $\CC$ has to be necessarily $(q-1)$-divisible). 
	Therefore, we have that  $A(t_0,\ldots,t_{q-1})=A_{w(c)}$ if $t_0=n-w(c)$ for some $c\in \CC$ and $t_i=t_j$ for all $i,j>0$, 
	and $A(t_0,\ldots,t_{q-1})=0$ otherwise.
\end{proof}

As a direct consequence of the lemma, we obtain the complete weight enumerator of $\CC_{\mathcal{L}}$.

\begin{prop} \label{cweCL}
Let $q$ be a prime power, $m$ a non-negative integer and $\mathcal{L} = \langle x^{q^{\ell_1}}, x^{q^{\ell_2}}, \ldots, x^{q^{\ell_s}} \rangle$ an ideal in $\mathbb{F}_{q^m}[x]$ such that $1 \le \ell_1 < \ell_2 < \cdots < \ell_s < \frac m2$. If $\mathcal{L}$ is an even rank family then the complete weight enumerator of $\CC_{\mathcal{L}}$ is given by
	$$W_{\CC_\mathcal{L}}(z_0,\ldots,z_{q-1}) = z_{0}^{n} + 
	\sum_{i=1}^{2} \sum_{r\in R_{\mathcal{L}}} M_{r,i} \; z_{0}^{a(r,i)}z_{1}^{b(r,i)}\cdots z_{q-1}^{b(r,i)}$$
	where $M_{r,i}$ and $R_{\mathcal{L}}$ are as in \eqref{M1r} and \eqref{LR} and
	$$a(r,i)=q^{m-1}+(-1)^{i+1}(q-1)q^{m-\frac{r}{2}-1}-1, \qquad 	b(r,i)=q^{m-1}+(-1)^{i} q^{m-\frac{r}{2}-1}.$$ 
\end{prop}

\section{The codes associated to $x^{q^\ell +1}$}
Here, we consider the codes $\CC_\mathcal{L}$, $\CC_{\mathcal{L},0}$, $\CC_{\mathcal{L},1}$ and $\CC_{\mathcal{L},2}$ from the previous section but in the particular case of $\mathcal{L} = \langle x^{q\ell} \rangle$, that we denote by $\CC_\ell$, $\CC_{\ell,0}$, 
$\CC_{\ell,1}$ and $\CC_{\ell,2}$. 
We will compute the spectra of these codes using Theorems \ref{Spec C_L} and \ref{Spec C_L,1} and Tables 1--4, but we explicitly compute the rank distribution in $\LL$ and their associated numbers $M_{r,i}$.

\subsubsection*{The codes $\CC_{\ell}$ and $\CC_{\ell,0}$}
Consider the irreducible cyclic code $\CC_{\ell}$ and the code $\CC_{\ell,0}$ over $\mathbb{F}_{q}$, 
with check polynomial $h_{\ell}(x)$ and $h_{\ell}(x)(x-1)$, respectively, where $h_{\ell}$ is the minimal polynomial of 
$\alpha^{-(q^\ell+1)}$, with $\alpha$ a primitive element. 
By Delsarte's Theorem these codes can be described by 
\begin{equation} \label{Cl0}
\begin{split}
\CC_{\ell} & = \Big \{c(\gamma)= \big( \tr_{q^m/q}(\gamma\alpha^{(q^\ell+1)i}) \big)_{i=0}^{n-1} :\gamma\in\mathbb{F}_{q^m} \Big\}, \\
\CC_{\ell,0} & = \Big \{ c_{b}(\gamma) = \big( \tr_{q^m/q}(\gamma\alpha^{(q^\ell+1)i}) +b\big)_{i=0}^{n-1} : 
\gamma \in \mathbb{F}_{q^m},b \in \ff_q \Big \}.
\end{split}
\end{equation}
Note that $c_0(\gamma)=c(\gamma)$ and that $\CC_\ell \subset \CC_{\ell,0}$.

Now we give the parameters and the spectra of these codes. 

\begin{thm}\label{DistCl}
Let $q$ be a prime power and $m, \ell$ positive integers such that {\blue $\ell < \tfrac m2$ and} $m_{\ell}= \frac{m}{(m,\ell)}$ is even. 
Then, $\CC_{\ell}$ is a $[n,m,d]_q$-code with $n=\tfrac{q^m-1}{D}$ and $d=\tfrac 1D q^{\frac m2-1} (q-1) d'$ where $D=q^{(m,\ell)}+1$ and
	$$d'= \begin{cases} 
	q^{\frac m2}-1 				 & \qquad \text{if $\frac 12 m_\ell$ even}, \\[1.5mm]
	q^{\frac m2}-q^{(m,\ell)}		 & \qquad \text{if $\frac 12 m_\ell$ odd}.
	\end{cases}$$
	On the other hand, $\CC_{\ell,0}$ is a $[n,m+1,\hat{d}]_q$-code with  
	$\hat{d} = \tfrac 1D (q^{m-1}(q-1) - \bar{d})$ 
	and 
	$$\bar{d}= \begin{cases}
 				 q^{\frac{m}{2}+(m,\ell)-1}+1 & \qquad \text{if $\tfrac 12 m_{\ell}$ even}, \\[1,5mm] 
				q^{\frac{m}{2}+(m,\ell)-1}(q-1) & \qquad \text{if $\tfrac 12 m_{\ell}$ odd}.
		\end{cases}$$  
The weight distributions of $\CC_{\ell}$ and $\CC_{\ell,0}$ are given by Tables 5 and 6 below.
\end{thm}

\begin{table}[h]
	\centering
	\caption{Weight distribution of $\CC_{\ell}$.} 
	\medskip
	\begin{tabular}{|c|c|c|}
		\hline 
		weight & frequency \\ 
		\hline \hline
		$0$ & $1$ \\
		\hline
		$\tfrac{1}{D} \big \{ q^m-q^{m-1}+ (-1)^{\frac 12 m_{\ell}}(q-1)q^{\frac{m}{2}+(m,\ell)-1} \big \}$ 
		& $n$  \\ 
		\hline 
		$\tfrac 1D \big \{ q^m-q^{m-1}+(-1)^{\frac 12 m_{\ell}+1}(q-1)q^{\frac{m}{2}-1} \big \}$ 
		& $ nq^{(m,\ell)}$ \\ 
		\hline 
	\end{tabular}
\end{table}

\renewcommand*{\arraystretch}{1} 
\begin{table}[h]
	\centering
	\caption{Weight distribution of $\CC_{\ell,0}$.} 
	\medskip
	\begin{tabular}{|c|c|c|}
		\hline 
		weight & frequency \\ 
		\hline \hline
		$0$ & $1$ \\
		\hline
		$\frac{q^m-1}{D}$ 
		&  $q-1$ \\ 
		\hline 
		$\frac 1D \big \{ q^m-q^{m-1}+ (-1)^{\frac 12 m_{\ell}}(q-1)q^{\frac{m}{2}+(m,\ell)-1} \big \} $  
		& $n$  \\ 
		\hline 
		$\tfrac 1D \big \{ q^m-q^{m-1}+(-1)^{\frac 12 m_{\ell}+1}(q-1)q^{\frac{m}{2}-1} \big \}$  
		& $ nq^{(m,\ell)}$ \\ 
		\hline 
		$\tfrac 1D \big \{ q^m-q^{m-1}+(-1)^{\frac 12 m_{\ell}+1}q^{\frac{m}{2}+(m,\ell)-1}-1 \big \} $  
		& $n(q-1)$ \\ 
		\hline 
		$\frac 1D \big \{q^m-q^{m-1}+(-1)^{\frac 12 m_{\ell}}q^{\frac{m}{2}-1}-1 \big \} $ 
		& $nq^{(m,\ell)}(q-1)$ \\ 
		\hline 
	\end{tabular}
\end{table}

\begin{proof}
Let us begin by computing the length $n$ of these codes. 
Since $m_{\ell}$ is even, by Lemma~\ref{mcd} we have that $n=\frac{q^m-1}{q^{(m,\ell)}+1}$. Thus $q^{(m,\ell)}+1\mid q^m-1$ and by 
Lemma~\ref{LemaEmes} we have $n=M$ or $n=M_1=M_2$ in even or odd characteristic respectively,
where $M$, $M_{1}$ and $M_{2}$ are the cardinalities of the sets $S_{q,m}(\ell)$, $X_{q,m}(\ell)$ and $Y_{q,m}(\ell)$ defined in 
\eqref{Sqml} and \eqref{Xqml}. 
Notice that in this case $q^m-1-M=nq^{(m,\ell)}$ in even characteristic and $q^m-1-M_1=q^m-1-M_2=nq^{(m,\ell)}$ in odd characteristic.
	
Let $\mathcal{L}_{\ell} = \langle x^{q^{\ell}} \rangle $, then 
$$R\in\mathcal{L}_{\ell} \qquad \Leftrightarrow \qquad R(x)=\gamma x^{q^\ell} = R_{\gamma}(x) \quad 
\text{ for some } \gamma\in\ff_{q^m}.$$ Thus $ Q_{R}(x)=\tr_{q^m/q}(\gamma x^{q^{\ell}+1})= Q_{\gamma,\ell}(x)$. 

Notice that the codes $\cod_{\mathcal{L}_{\ell}}$ and $\cod_{\mathcal{L}_{\ell},0}$ as in \eqref{codigoCL}, are obtained 
from $(q^m-1,q^{\ell}+1)$-copies of the codes $\cod_{\ell}$ and $\cod_{\ell,0}$ in \eqref{Cl0}, respectively. 
In terms of weights, this mean that 
$$ w(c_{b}(\gamma)) = \frac{w(c_{R_{\gamma},b})}{(q^m-1,q^{\ell}+1)}.$$
	
By Theorems \ref{Thmpar} and \ref{Thmimpar}, $\mathcal{L}_{\ell}$ is an even rank family. 
Furthermore, $R_{\mathcal{L}_{\ell}}=\{0, m, \, m-2(m,\ell)\}.$ If $q$ is even, by Theorem \ref{Thmpar} we have that 
$$\begin{tabular}{llll}
$M_{m,2} = M_{m-2(m,\ell),1} = 0$, & \quad $M_{m,1}=nq^{(m,\ell)}$, &\quad $M_{m-2(m,\ell),2}=n$, &\qquad  if $\tfrac{1}{2} m_{\ell}$ even, 
	\\[2mm]
	$M_{m,1}=M_{m-2(m,\ell),2}=0$, & \quad $M_{m-2(m,\ell),1}=n$, & \quad $M_{m,2}=nq^{(m,\ell)}$,  & \qquad if $\tfrac{1}{2} m_{\ell}$ odd.
\end{tabular}$$
Similarly, if $q$ is odd, Theorem \ref{Thmimpar} implies that 
	$$\begin{tabular}{llll}
	$M_{m,2} = M_{m-2(m,\ell),1} = 0,$ &\quad $M_{m,1}=nq^{(m,\ell)}$, &\quad $M_{m-2(m,\ell),2}=n$, &\qquad  if $\tfrac{1}{2} m_{\ell}$ even, 
	\\[2mm]
	$M_{m,1}=M_{m-2(m,\ell),2}=0$, &\quad $M_{m-2(m,\ell),1}=n$, &\quad $M_{m,2}=nq^{(m,\ell)}$, & \qquad  if $\tfrac{1}{2} m_{\ell}$ odd.
	\end{tabular}$$
Now, by Theorem \ref{Spec C_L}, we obtain the weights and frequencies given in Tables 5 and 6. 
Finally, by studying the values in the tables, we get the minimal distances for both codes. 
\end{proof}

	%

\begin{coro} \label{cwe Cl}
	Under the same hypothesis as before, the complete weight enumerator of $C_{\ell}$ is 
		$$W_{\CC_{\ell}}(z_0,\ldots,z_{q-1}) = 
		z_{0}^{n} + n z_{0}^{a_0} z_{1}^{a_{1}} \cdots z_{q-1}^{a_{1}}+nq^{(m,\ell)}z_{0}^{a'_0}z_{1}^{a'_{1}}\cdots z_{q-1}^{a'_{1}}$$
		where
		$$\begin{array}{ll}
		a_0 = n-\tfrac{(q-1)q^{m-1}}{q^{(m,\ell)}+1} (1 + (-1)^{\frac 12 m_\ell} \, q^{{(m,\ell)}-\frac{m}{2}}), & \qquad 
		a_1 =	\tfrac{q^{m-1}}{q^{(m,\ell)}+1}(1 + (-1)^{\frac 12 m_\ell} \, q^{{(m,\ell)}-\frac{m}{2}}), \medskip \\[2mm]
		a'_0 = n-\tfrac{(q-1)q^{m-1}}{q^{(m,\ell)}+1}(1 +(-1)^{\frac{1}{2}m_{\ell}+1}\, q^{-\frac{m}{2}}), & \qquad
		a'_1 = \tfrac{q^{m-1}}{q^{(m,\ell)}+1}(1 +(-1)^{\frac{1}{2}m_{\ell}+1} \, q^{-\frac{m}{2}}).
		\end{array}$$
\end{coro}

\begin{exam} \label{ej 43}
	Let $q=2$, $m=8$ and $\ell=1$. 
By Theorems \ref{Spec C_L} and \ref{Spec C_L,1} the codes $\CC_\ell$ and $\CC_{\ell,0}$ have paremeters $[85,8,40]$ and $[85,9,37]$ respectively, with weight enumerators given by 
\begin{align*}
W_{\CC_{\ell}}(x) = 1+ 170 x^{120} + 85x^{144} \quad \text{and} \quad 
W_{\CC_{\ell,0}}(x) = 1+ 85 x^{37} + 170x^{40} + 170x^{45} + 85x^{48} + x^{85}.
\end{align*}

\end{exam}

\subsubsection*{The codes $\CC_{\ell,1}$ and $\CC_{\ell,2}$}
Consider the codes $\CC_{\ell,1}$ and $\CC_{\ell,2}$ over $\mathbb{F}_{q}$, 
with check polynomials $h_{\ell}(x)h_1(x)$ and $h_{\ell}(x)h_1(x)(x-1)$, respectively. Here, $h_{\ell}$ and $h_{1}(x)$ are the minimal polynomials of $\alpha^{-(q^\ell+1)}$ and $\alpha^{-1}$ respectively, where $\alpha$ is a primitive element of $\ff_{q^m}$. 
By Delsarte's Theorem, these codes are given by 
\begin{equation} \label{Cl1}
\begin{split}
\CC_{\ell,1} &= \Big \{c(\gamma,\beta)= \big( \tr_{q^m/q}(\gamma x^{q^\ell+1}+\beta x) \big)_{x\in\ff_{q^m}^*} : 
\gamma,\beta\in\mathbb{F}_{q^m} \Big\}, \\
\CC_{\ell,2} &= \Big \{c_{b}(\gamma,\beta)= \big( \tr_{q^m/q}(\gamma x^{q^\ell+1}\beta x) +b\big)_{x\in\ff_{q^m}^*} :\gamma,\beta\in\mathbb{F}_{q^m},b\in\ff_q \Big\}.
\end{split}
\end{equation}
As before, for $m,\ell$ positive integers such that $m/(m,\ell)$ even we denote 
$n=\tfrac{q^m-1}{q^{(m,\ell)}+1}.$
We now give the parameters and the spectra of these codes.

\begin{thm}\label{Dist Cl1}
Let $q$ be a prime power and $m, \ell$ positive integers such that $m_\ell$ is even. 
Then, $\CC_{\ell,1}$ is a $[N,2m,d]_q$-code with $N=q^m-1$ and $d=q^{m-1}(q-1)-d'$ with
$$d'= \begin{cases} 
 q^{\frac{m}{2}+(m,\ell)-1} 				 & \qquad \text{if $\frac 12 m_\ell$ even}, \\[1.5mm]
 (q-1)q^{\frac{m}{2}+(m,\ell)-1}		 & \qquad \text{if $\frac 12 m_\ell$ odd}
\end{cases}$$
and $\CC_{\ell,2}$ is a $[N,2m+1,d-1]_q$-code. 
The weight distributions of the codes 
$\CC_{\ell,1}$ and $\CC_{\ell,2}$ are given by Tables 7 and 8 below. 
\end{thm}

\renewcommand*{\arraystretch}{1}
\begin{table}[!ht]
	\centering
	\caption{ Weight distribution of $\CC_{\ell,1}$.} 
	\medskip
	\resizebox{16cm}{!}{\begin{tabular}{|c|c|c|}
			\hline 
			weight & frequency \\ 
			\hline \hline
			0& 1 \\
			\hline
			$q^m-q^{m-1}$ &  $n(q^m-q^{m-2{(m,\ell)}}) + q^m-1$ \\ 
			\hline 
			$q^m-q^{m-1}+(-1)^{\frac{1}{2}m_{\ell}+1}(q-1)q^{\frac{m}{2}-1}$ &  $nq^{(m,\ell)} (q^{m-1}+(-1)^{\frac{1}{2}m_{\ell}}(q-1)q^{\frac{m}{2}-1})$ \\ 
			\hline 
			$q^m-q^{m-1}+(-1)^{\frac{1}{2}m_{\ell}}(q-1)q^{\frac{m}{2}+{(m,\ell)}-1}$  & $n(q^{m-1-2{(m,\ell)}}+(-1)^{\frac{1}{2}m_{\ell}+1}(q-1)q^{\frac{m}{2}-{(m,\ell)}-1})$ \\ 
			\hline  
			$q^m-q^{m-1}+(-1)^{\frac{1}{2}m_{\ell}}q^{\frac{m}{2}-1}$ & $nq^{(m,\ell)}(q^{m-1}+(-1)^{\frac{1}{2}m_{\ell}+1}q^{\frac{m}{2}-1})(q-1)$ \\ 
			\hline
			$q^m-q^{m-1}+(-1)^{\frac{1}{2}m_{\ell}+1}q^{\frac{m}{2}+{(m,\ell)}-1}$ & $n(q^{m-1-2{(m,\ell)}}+(-1)^{\frac{1}{2}m_{\ell}}q^{\frac{m}{2}-{(m,\ell)}-1})(q-1)$ \\ 
			\hline 
			
	\end{tabular}}
\end{table}

\renewcommand*{\arraystretch}{1}
\begin{table}[!ht]
	\centering
	\caption{Weight distribution of $\CC_{\ell,2}$.} 
	\medskip
	\resizebox{16cm}{!}{\begin{tabular}{|c|c|c|}
			\hline 
			weight & frequency \\ 
			\hline \hline
			0& 1 \\
			\hline
			$q^m-q^{m-1}-1$ &  $n(q^m-q^{m-2{(m,\ell)}})(q-1)+(q^m-1)(q-1)$ \\ 
			\hline 
			$q^m-q^{m-1}+(-1)^{\frac{1}{2}m_{\ell}+1}(q-1)q^{\frac{m}{2}-1}-1$ &  $nq^{(m,\ell)}(q^{m-1}+(-1)^{\frac{1}{2}m_{\ell}+1}q^{\frac{m}{2}-1})(q-1)$ \\ 
			\hline 
			$q^m-q^{m-1}+(-1)^{\frac{1}{2}m_{\ell}}(q-1)q^{\frac{m}{2}+{(m,\ell)}-1}-1$  & $n(q^{m-1-2{(m,\ell)}}+(-1)^{\frac{1}{2}m_{\ell}}q^{\frac{m}{2}-{(m,\ell)}-1})(q-1)$ \\ 
			\hline 
			$q^m-q^{m-1}+(-1)^{\frac{1}{2}m_{\ell}}q^{\frac{m}{2}-1}-1$ & $nq^{(m,\ell)}(q^m-q^{m-1}+(-1)^{\frac{1}{2}m_{\ell}}q^{\frac{m}{2}-1})(q-1)$ \\ 
			\hline 
			$q^m-q^{m-1}+(-1)^{\frac{1}{2}m_{\ell}+1}q^{\frac{m}{2}+{(m,\ell)}-1}-1$ & $n(q^{m-2{(m,\ell)}}-q^{m-1-2{(m,\ell)}}+(-1)^{\frac{1}{2}m_{\ell}+1}q^{\frac{m}{2}-{(m,\ell)}-1})(q-1)$ \\ 
			\hline
			$q^m-1$ & $q-1$\\
			\hline 
			$q^m-q^{m-1}$ &  $n(q^m-q^{m-2{(m,\ell)}}) + q^m-1$ \\ 
			\hline 
			$q^m-q^{m-1}+(-1)^{\frac{1}{2}m_{\ell}+1}(q-1)q^{\frac{m}{2}-1}$ &  $nq^{(m,\ell)} (q^{m-1}+(-1)^{\frac{1}{2}m_{\ell}}(q-1)q^{\frac{m}{2}-1})$ \\ 
			\hline 
			$q^m-q^{m-1}+(-1)^{\frac{1}{2}m_{\ell}}(q-1)q^{\frac{m}{2}+{(m,\ell)}-1}$  & $n(q^{m-1-2{(m,\ell)}}+(-1)^{\frac{1}{2}m_{\ell}+1}(q-1)q^{\frac{m}{2}-{(m,\ell)}-1})$ \\ 
			\hline  
			$q^m-q^{m-1}+(-1)^{\frac{1}{2}m_{\ell}}q^{\frac{m}{2}-1}$ & $nq^{(m,\ell)}(q^{m-1}+(-1)^{\frac{1}{2}m_{\ell}+1}q^{\frac{m}{2}-1})(q-1)$ \\ 
			\hline
			$q^m-q^{m-1}+(-1)^{\frac{1}{2}m_{\ell}+1}q^{\frac{m}{2}+{(m,\ell)}-1}$ & $n(q^{m-1-2{(m,\ell)}}+(-1)^{\frac{1}{2}m_{\ell}}q^{\frac{m}{2}-{(m,\ell)}-1})(q-1)$ \\ 
			\hline 
	\end{tabular}}
\end{table}

\begin{proof}
Note that $\CC_{\ell, 1}=\CC_{\mathcal{L}_{\ell},1}$ and $\CC_{\ell,2} = \CC_{\mathcal{L}_{\ell},2}$ with 
$\mathcal{L}_{\ell} = \langle x^{q^{\ell}}\rangle$ where $\CC_{\mathcal{L}_{\ell},1}, \CC_{\mathcal{L}_{\ell},2}$ are the codes defined in \eqref{codigoCL0}. Then, by Theorem \ref{Spec C_L,1}, it is enough to compute the numbers $M_{r,1}$, $M_{r,2}$ and the set 
$R_{\mathcal{L}_{\ell}}$. They have been calculated in the proof of the Theorem \ref{DistCl}. Therefore, the Tables 7 and 8 give us the spectra of the codes $\CC_{\ell,1}$ and $\CC_{\ell,2}$ as we wanted.
\end{proof}

\begin{exam}
	Let $q=2$, $m=8$ and $\ell=1$ as in Example \ref{ej 43}. 
	By Theorem \ref{Dist Cl1}, the codes $\CC_{\ell,1}$ and $\CC_{\ell,2}$ have paremeters $[255,16,112]$ and $[255,17,111]$, respectively. Also, we have 
\begin{align*}
& W_{\CC_{\ell,1}}(x) = 1 + 3060 \, x^{112} + 23120 \, x^{120} + 16575 \, x^{128} + 20400 \, x^{136} + 2380 \, x^{144} \\
& \begin{multlined}
W_{\CC_{\ell,2}}(x) = 1 + 2380 \, x^{111} + 3060 \, x^{112} + 20400 \, x^{119} + 23120 \, x^{120} + 16575 \, x^{127} \\  + 16575 \, x^{128} + 23120 \, x^{135} + 20400\, x^{136} + 3060 \, x^{143}  + 2380 \, x^{144} + x^{255}.
\end{multlined}
\end{align*}
\end{exam}

\begin{rem}
($i$) From Tables 5--8 we see that $\CC_\ell$ is a $2$-weight code, $\CC_{\ell,0}$ and $\CC_{\ell,1}$ are 
$5$-weight codes and $\CC_{\ell,2}$ is an $11$-weight code. 
Also, one checks that $\CC_\ell$ is $q^{\frac m2 -1}(q-1)$-divisible and $\CC_{\ell,1}$ is $q^{\frac m2 -1}$-divisible.
These facts are in accordance with Klapper's Theorems \ref{Thmpar} and \ref{Thmimpar} and Remark~\ref{rho weight}.

($ii$) In the binary case (i.e.\@ $q=2$), the codes $\CC_{\LL,0}$ and $\CC_{\LL,2}$ have symmetric spectrum, that is $A_i=A_{n-i}$ for every $i$, since the word $11 \cdots 11$ is in these codes (there is a weight $w=n$). 
\end{rem}

\begin{rem}
It can be shown, via Pless power moments, that if $q=2$ and $(m,\ell)=1$ the dual code of $\CC_{\ell,1}$ is optimal in the sense that its minimal distance is maximum in the class of cyclic codes with generator polynomial $m_{\alpha}(x)m_{\alpha^{t}}(x)$ over 
$\ff_{2}$. This condition of optimality is equivalent to the function $f(x)=x^t$ defined over $\ff_{2^m}$ 
being an APN function (see \cite{Ch}). In our case, $f_{\ell}(x) = x^{2^{\ell}+1}$ with $(m,\ell)=1$, is a well-known APN function, namely the Kasami--Gold function.
\end{rem}

\section{Codes associated to $\mathcal{L}_{\ell,3\ell}$}
In this section we consider the codes $\CC_{\LL}$, $\CC_{\LL,0}$, $\CC_{\LL,1}$ and $\CC_{\LL,2}$ associated to the family of 
$p$-linearized polynomials 
$$\LL = \mathcal{L}_{\ell,3 \ell} = \langle x^{p^\ell}, x^{p^{3\ell}} \rangle \subset \mathbb{F}_{p^{m}}[x]$$ 
where $p$ is an odd prime and $m_{\ell} = m/(m,\ell)$ even. The next theorem summarizes, in our notation, the results proved in 
\cite{ZWZH1}. 

\begin{thm}[\cite{ZWZH1}] \label{ranks l3l}
Let $p$ be an odd prime and let $m,\ell$ be non-negative integers such that $m_{\ell}=m/(m,\ell)$ is even with $m>6\ell$ and 
denote $\delta=(m,\ell)$. 
Then, $\mathcal{L}_{\ell,3\ell}$ is an even rank family with 
$R_{\mathcal{L}_{\ell,3\ell}}=\{m, m-2\delta,m-4\delta,m-6\delta\}$ $($see \eqref{LR}$)$.
Moreover, the numbers $M_{r,i}$, as defined in \eqref{M1r}, 
have the following expressions: 
\begin{enumerate}[(a)]
\item If $\tfrac{1}{2} m_{\ell}$ is odd, then $M_{m,1} = M_{m-2\delta,2} = M_{m-4\delta,1} = M_{m-6\delta,2} = 0$ and 
 \begin{align*}
		M_{m,2} &= \tfrac{p^{2m+6\delta}-p^{2m+4\delta}-p^{2m+\delta}+p^{m+4\delta}+p^{m+\delta}-p^{6\delta}-p^{\frac{3m}{2}+5\delta}+p^{\frac{3m}{2}+4\delta}+p^{\frac{m}{2}+5\delta}-p^{\frac{m}{2}+4\delta}}{p^{6\delta}+p^{5\delta}-p^{4\delta}+p^{2\delta}-p^{\delta}-1}, \\[2mm]
		M_{m-2\delta,1} & = \tfrac{p^{2m-2\delta}(p^{7\delta}-p^{2\delta}-1)+p^{m-2\delta}(p^{5\delta}-p^{6\delta}+p^{2\delta}+1)-p^{3\delta}(p^{2\delta}-p^{\delta}+1)+(p^{\frac{3m}{2}}-p^{\frac{m}{2}})(\sum_{i=0}^{5}(-1)^{i+1}p^{i\delta})}{p^{6\delta}+p^{5\delta}-p^{4\delta}+p^{2\delta}-p^{\delta}-1}, \\[2mm]
		M_{m-4\delta,2} & = \tfrac{p^{2m-3\delta}(p^{5\delta}+p^{\delta}-1)-p^{m-3\delta}(p^{6\delta}+p^{4\delta}+p^{\delta}-1)+p^{\delta}(p^{2\delta}-p^{\delta}+1)-(p^{\frac{3m}{2}-2\delta}-p^{\frac{m}{2}-2\delta})(\sum_{i=0}^{5}(-1)^{i+1}p^{i\delta})}{p^{6\delta}+p^{5\delta}-p^{4\delta}+p^{2\delta}-p^{\delta}-1}, \\[2mm]
		M_{m-6\delta,1} & = \tfrac{p^{2m-3\delta}-p^{m}-p^{m-3\delta}+1+p^{\frac{3m}{2}-\delta}-p^{\frac{3m}{2}-2\delta}-p^{\frac{m}{2}-\delta}+p^{\frac{m}{2}-2\delta}}{p^{6\delta}+p^{5\delta}-p^{4\delta}+p^{2\delta}-p^{\delta}-1}.
		\end{align*}

\item If $\tfrac{1}{2} m_{\ell}$ is even, then $M_{m,2}=M_{m-2\delta,1}=M_{m-4\delta,2}=M_{m-6\delta,1}=0$ and
		\begin{align*}
		M_{m,2} & = \tfrac{p^{2m+6\delta}-p^{2m+4\delta}-p^{2m+\delta}+p^{m+4\delta}+p^{m+\delta}-p^{6\delta}+p^{\frac{3m}{2}+5\delta}-p^{\frac{3m}{2}+4\delta}-p^{\frac{m}{2}+5\delta}+p^{\frac{m}{2}+4\delta}}{p^{6\delta}+p^{5\delta}-p^{4\delta}+p^{2\delta}-p^{\delta}-1}, 
		\\[2mm]
		M_{m-2\delta,2} & = \tfrac{p^{2m-2\delta}(p^{7\delta}-p^{2\delta}-1)+p^{m-2\delta}(p^{5\delta}-p^{6\delta}+p^{2\delta}+1)-p^{3\delta}(p^{2\delta}-p^{\delta}+1)-(p^{\frac{3m}{2}}-p^{\frac{m}{2}})(\sum_{i=0}^{5}(-1)^{i+1}p^{i\delta})}{p^{6\delta}+p^{5\delta}-p^{4\delta}+p^{2\delta}-p^{\delta}-1}, \\[2mm]
		M_{m-4\delta,1} & = \tfrac{p^{2m-3\delta}(p^{5\delta}+p^{\delta}-1)-p^{m-3\delta}(p^{6\delta}+p^{4\delta}+p^{\delta}-1)+p^{\delta}(p^{2\delta}-p^{\delta}+1)+(p^{\frac{3m}{2}-2\delta}-p^{\frac{m}{2}-2\delta})(\sum_{i=0}^{5}(-1)^{i+1}p^{i\delta})}{p^{6\delta}+p^{5\delta}-p^{4\delta}+p^{2\delta}-p^{\delta}-1}, \\[2mm]
		M_{m-6d,2} & = \tfrac{p^{2m-3\delta}-p^{m}-p^{m-3\delta}+1-p^{\frac{3m}{2}-\delta}+p^{\frac{3m}{2}-2\delta}+p^{\frac{m}{2}-\delta}-p^{\frac{m}{2}-2\delta}}{p^{6\delta}+p^{5\delta}-p^{4\delta}+p^{2\delta}-p^{\delta}-1}.
		\end{align*}
	\end{enumerate}
\end{thm}

In \cite{ZWZH1}, the distribution of ranks and types given in the previous theorem was used to calculate the spectra of the codes 
$\CC_{\LL}$ and $\CC_{\LL,1}$ with $\LL=\LL_{\ell,3\ell}$. 
Fortunately, this information is enough to calculate the spectra of $\CC_{\LL,0}$ and $\CC_{\LL,2}$ also, 
which follows directly from Theorems \ref{Spec C_L}, \ref{Spec C_L,1} and \ref{ranks l3l}.

\begin{thm} \label{spec cl3l}
Let $p$ be an odd prime and let $m,\ell$ be positive integers such that $m_{\ell}=m/(m,\ell)$ is even with $m>6\ell$.
Then, $\CC_{\mathcal{L}_{\ell,3\ell},0}$ is a $[n,2m+1,d]_p$-code with $n=p^m-1$ and $d=p^{m-1}(p-1)-d'$ where 
$$d'= \begin{cases} 
p^{\frac{m}{2}+3(m,\ell)-1}+1 				 & \qquad \text{if $\frac 12 m_\ell$ even}, \\[1.5mm]
(p-1)p^{\frac{m}{2}+3(m,\ell)-1}		 & \qquad \text{if $\frac 12 m_\ell$ odd},
\end{cases}$$
and $\CC_{\mathcal{L}_{\ell,3\ell},2}$ is a $[n,3m+1,\hat{d}]_p$-code with $\hat d = d$ if $\frac 12 m_\ell$ is even and 
$\hat d = d-1$ if $\frac 12 m_\ell$ is odd. 
The weight distributions of the codes $\CC_{\mathcal{L}_{\ell,3\ell},0}$ and $\CC_{\mathcal{L}_{\ell,3\ell},2}$
are given by Tables 9 and 10 below. 
\end{thm}

We set these notations for the next two tables:
\begin{equation} \label{Ris}
\begin{split}
& F_0 = \tfrac{p^{2m+6\delta}-p^{2m+4\delta}-p^{2m+\delta}+p^{m+4\delta}+p^{m+\delta}-p^{6\delta}+\varepsilon_{\ell}(p^{\frac{3m}{2}+5\delta}-p^{\frac{3m}{2}+4\delta}-p^{\frac{m}{2}+5\delta}+p^{\frac{m}{2}+4\delta})}{p^{6\delta}+p^{5\delta}-p^{4\delta}+p^{2\delta}-p^{\delta}-1}, \\[2mm]
& F_1 = \tfrac{p^{2m-2\delta}(p^{7\delta}-p^{2\delta}-1)+p^{m-2\delta}(p^{5\delta}-p^{6\delta}+p^{2\delta}+1)-p^{3\delta}(p^{2\delta}-p^{\delta}+1)-\varepsilon_{\ell}(p^{\frac{3m}{2}}-p^{\frac{m}{2}})(\sum_{i=0}^{5}(-1)^{i+1}p^{i\delta})}{p^{6\delta}+p^{5\delta}-p^{4\delta}+p^{2\delta}-p^{\delta}-1}, \\[2mm]
& F_2 = \tfrac{p^{2m-3\delta}(p^{5\delta}+p^{\delta}-1)-p^{m-3\delta}(p^{6\delta}+p^{4\delta}+p^{\delta}-1)+p^{\delta}(p^{2\delta}-p^{\delta}+1)+\varepsilon_{\ell}(p^{\frac{3m}{2}-2\delta}-p^{\frac{m}{2}-2\delta})(\sum_{i=0}^{5}(-1)^{i+1}p^{i\delta})}{p^{6\delta}+p^{5\delta}-p^{4\delta}+p^{2\delta}-p^{\delta}-1}, \\[2mm]
& F_3 = \tfrac{p^{2m-3\delta}-p^{m}-p^{m-3\delta}+1-\varepsilon_{\ell}(p^{\frac{3m}{2}-\delta}-p^{\frac{3m}{2}-2\delta}-p^{\frac{m}{2}-\delta}+p^{\frac{m}{2}-2\delta})}{p^{6\delta}+p^{5\delta}-p^{4\delta}+p^{2\delta}-p^{\delta}-1}. 
\end{split}
\end{equation}
\begin{table}[H]
	\centering
	\caption{ Weight distribution of $\CC_{\mathcal{L}_{\ell,3\ell},0}$.} 
	\medskip
	\begin{tabular}{|c|c|c|}
		\hline 
		weight & frequency \\ 
		\hline \hline
		$0$ & $1$ \\
		\hline
		$p^m-p^{m-1}+(-1)^{\frac{m_{\ell}}{2}+1}(p-1)p^{\frac{m}{2}-1}$ & $F_0$  \\ 
		\hline
		$p^m-p^{m-1}+(-1)^{\frac{m_{\ell}}{2}}(p-1)p^{\frac{m}{2}+(m,\ell)-1}$ & $F_1$  \\
		\hline
		$p^m-p^{m-1}+(-1)^{\frac{m_{\ell}}{2}+1}(p-1)p^{\frac{m}{2}+2(m,\ell)-1}$ & $F_2$  \\
		\hline
		$p^m-p^{m-1}+(-1)^{\frac{m_{\ell}}{2}}(p-1)p^{\frac{m}{2}+3(m,\ell)-1}$ & $F_3$  \\
		\hline
		$p^m-1$ & $p-1$ \\
		\hline  
		$p^m-p^{m-1}+(-1)^{\frac{m_{\ell}}{2}}p^{\frac{m}{2}-1}-1$ & $(p-1)F_0$ \\
		\hline  
		$p^m-p^{m-1}+(-1)^{\frac{m_{\ell}}{2}+1}p^{\frac{m}{2}+(m,\ell)-1}-1$ & $(p-1)F_1$ \\
		\hline  
		$p^m-p^{m-1}+(-1)^{\frac{m_{\ell}}{2}}p^{\frac{m}{2}+2(m,\ell)-1}-1$ & $(p-1)F_2$ \\
		\hline  
		$p^m-p^{m-1}+(-1)^{\frac{m_{\ell}}{2}+1}p^{\frac{m}{2}+3(m,\ell)-1}-1$ & $(p-1)F_3$ \\
		\hline
	\end{tabular}
\end{table}

\begin{table}[h!]
	\centering
	\caption{ Weight distribution of $\CC_{\mathcal{L}_{\ell,3\ell},2}$.} 
	\medskip
	\resizebox{16cm}{!}{\begin{tabular}{|c|c|c|}
			\hline 
			weight & frequency \\ 
			\hline \hline
			$0$ & $1$ \\
			\hline
			$p^m-p^{m-1}$ &  $p^m-1+\sum_{i=0}^{3}R_{i}(p^m-p^{m-i(m,\ell)})$ \\ 
			\hline 
			$p^m-p^{m-1}+(-1)^{\frac{m_{\ell}}{2}+1}(p-1)p^{\frac{m}{2}-1}$ & $(p^{m-1}+(-1)^{\frac{m_{\ell}}{2}}(p-1)p^{\frac{m}{2}-1}) F_0$ \\
			\hline 
			$p^m-p^{m-1}+(-1)^{\frac{m_{\ell}}{2}}(p-1)p^{\frac{m}{2}+(m,\ell)-1}$ & $(p^{m-2(m,\ell)-1}+(-1)^{\frac{m_{\ell}}{2}+1}(p-1)p^{\frac{m}{2}-(m,\ell)-1})F_1$  \\
			\hline 
			$p^m-p^{m-1}+(-1)^{\frac{m_{\ell}}{2}+1}(p-1)p^{\frac{m}{2}+2(m,\ell)-1}$ & $(p^{m-4(m,\ell)-1}+(-1)^{\frac{m_{\ell}}{2}}(p-1)p^{\frac{m}{2}-2(m,\ell)-1})F_2$  \\
			\hline 
			$p^m-p^{m-1}+(-1)^{\frac{m_{\ell}}{2}}(p-1)p^{\frac{m}{2}+3(m,\ell)-1}$ & $(p^{m-6(m,\ell)-1}+(-1)^{\frac{m_{\ell}}{2}+1}(p-1)p^{\frac{m}{2}-3(m,\ell)-1})F_3$  \\ 
			\hline 
			$p^m-p^{m-1}+(-1)^{\frac{m_{\ell}}{2}}p^{\frac{m}{2}-1}$  & $(p^{m-1}+(-1)^{\frac{m_{\ell}}{2}+1}p^{\frac{m}{2}-1})(p-1)F_0$ \\
			\hline 
			$p^m-p^{m-1}+(-1)^{\frac{m_{\ell}}{2}+1}p^{\frac{m}{2}+(m,\ell)-1}$  & $(p^{m-2(m,\ell)-1}+(-1)^{\frac{m_{\ell}}{2}}p^{\frac{m}{2}-(m,\ell)-1})(p-1)F_1$ \\
			\hline 
			$p^m-p^{m-1}+(-1)^{\frac{m_{\ell}}{2}}p^{\frac{m}{2}+2(m,\ell)-1}$  & $(p^{m-4(m,\ell)-1}+(-1)^{\frac{m_{\ell}}{2}+1}p^{\frac{m}{2}-2(m,\ell)-1})(p-1)F_2$ \\
			\hline 
			$p^m-p^{m-1}+(-1)^{\frac{m_{\ell}}{2}+1}p^{\frac{m}{2}+3(m,\ell)-1}$  & $(p^{m-6(m,\ell)-1}+(-1)^{\frac{m_{\ell}}{2}}p^{\frac{m}{2}-3(m,\ell)-1})(p-1)F_3$ \\ 
			\hline  
			$p^m-1$ & $p-1$ \\
			\hline
			$p^m-p^{m-1}-1$ &  $(p-1)(p^m-1+\sum_{i=0}^{3}R_i(p^m-p^{m-i(m,\ell)}))$ \\ 
			\hline 
			$p^m-p^{m-1}-1+(-1)^{\frac{m_{\ell}}{2}+1}(p-1)p^{\frac{m}{2}-1}$ & $(p^{m-1}+(-1)^{\frac{m_{\ell}}{2}+1}p^{\frac{m}{2}-1})(p-1)F_0$  \\ 
			\hline 
			$p^m-p^{m-1}-1+(-1)^{\frac{m_{\ell}}{2}}(p-1)p^{\frac{m}{2}+(m,\ell)-1}$ & $(p^{m-2(m,\ell)-1}+(-1)^{\frac{m_{\ell}}{2}}p^{\frac{m}{2}-(m,\ell)-1})(p-1)F_1$  \\
			\hline 
			$p^m-p^{m-1}-1+(-1)^{\frac{m_{\ell}}{2}+1}(p-1)p^{\frac{m}{2}+2(m,\ell)-1}$ & $(p^{m-4(m,\ell)-1}+(-1)^{\frac{m_{\ell}}{2}+1}p^{\frac{m}{2}-2(m,\ell)-1})(p-1)F_2$  \\
			\hline 
			$p^m-p^{m-1}-1+(-1)^{\frac{m_{\ell}}{2}}(p-1)p^{\frac{m}{2}+3(m,\ell)-1}$ & $(p^{m-6(m,\ell)-1}+(-1)^{\frac{m_{\ell}}{2}}p^{\frac{m}{2}-3(m,\ell)-1})(p-1)F_3$  \\
			\hline 
			$p^m-p^{m-1}-1+(-1)^{\frac{m_{\ell}}{2}}p^{\frac{m}{2}-1}$  & $(p^{m}-p^{m-1}+(-1)^{\frac{m_{\ell}}{2}}p^{\frac{m}{2}-1})(p-1)F_0$ \\ 
			\hline 
			$p^m-p^{m-1}-1+(-1)^{\frac{m_{\ell}}{2}+1}p^{\frac{m}{2}+(m,\ell)-1}$  & $(p^{m-2(m,\ell)}-p^{m-2(m,\ell)-1}+(-1)^{\frac{m_{\ell}}{2}+1}p^{\frac{m}{2}-(m,\ell)-1})(p-1)F_1$ \\
			\hline 
			$p^m-p^{m-1}-1+(-1)^{\frac{m_{\ell}}{2}}p^{\frac{m}{2}+2(m,\ell)-1}$  & $(p^{m-4(m,\ell)}-p^{m-4(m,\ell)-1}+(-1)^{\frac{m_{\ell}}{2}}p^{\frac{m}{2}-2(m,\ell)-1})(p-1)F_2$ \\
			\hline 
			$p^m-p^{m-1}-1+(-1)^{\frac{m_{\ell}}{2}+1}p^{\frac{m}{2}+3(m,\ell)-1}$  & $(p^{m-6(m,\ell)}-p^{m-6(m,\ell)-1}+(-1)^{\frac{m_{\ell}}{2}+1}p^{\frac{m}{2}-3(m,\ell)-1})(p-1)F_3$ \\
			\hline  
	\end{tabular}}
\end{table}

\begin{rem}
The weight distributions of $\CC_{\mathcal{L}_{\ell,3\ell}}$ and $\CC_{\mathcal{L}_{\ell,3\ell},1}$ 
are determined by those of $\CC_{\mathcal{L}_{\ell,3\ell},0}$ and $\CC_{\mathcal{L}_{\ell,3\ell},2}$, respectively. 
More precisely, the weight distribution of $\CC_{\mathcal{L}_{\ell,3\ell}}$ is given by the first $5$ rows of Table 9, and the spectrum of $\CC_{\mathcal{L}_{\ell,3\ell},1}$ is given by the first $10$ rows of Table 10. Therefore, $\CC_{\mathcal{L}_{\ell,3\ell}}$ is a 
$4$-weight code, $\CC_{\mathcal{L}_{\ell,3\ell},0}$ and $\CC_{\mathcal{L}_{\ell,3\ell},1}$ are 9-weight codes and 
$\CC_{\mathcal{L}_{\ell,3\ell},2}$ is a $19$-weight code. 
\end{rem}

As a direct consequence of Proposition \ref{cweCL} we obtain the following.

\begin{coro}
Under the hypothesis of Theorem \ref{spec cl3l}, the complete weight enumerator of $\CC_{\mathcal{L}_{\ell,3\ell}}$ is given by
$$W_{\CC_{\mathcal{L}_{\ell,3\ell}}}(z_0,\ldots,z_{p-1}) = z_{0}^{p^m-1} + 
\sum_{i=0}^{3} F_i \, z_{0}^{a_i} z_{1}^{b_i} \cdots z_{p-1}^{b_i},$$
where, for each $i=0,\ldots, 3$, the numbers $F_i$ are given in \eqref{Ris} and 
$$a_i = p^{m-1}+(-1)^{i}\varepsilon_{\ell}(p-1)p^{\frac m2 +i(m,\ell)-1}-1 \quad \text{and} \quad 
  b_i = p^{m-1}+(-1)^{i+1}\varepsilon_{\ell}\,p^{\frac m2 +i(m,\ell)-1}.$$
\end{coro}

\begin{proof}
By the previous remark, the weight enumerator of $\CC$ is 
$W_{\CC}(x) = 1 + \sum_{i=0}^{3} R_{i}\, x^{c_i}$ where 
\begin{equation} \label{ci} 
c_i = (p-1)\big(p^{m-1}+(-1)^{i+1} \varepsilon_{\ell} \,p^{\frac m2 +i(m,\ell)+1}\big).
\end{equation} 
Thus, by Proposition \ref{cweCL}, we have 
$W_{\CC}(z_0,z_1,\ldots,z_{p-1}) = z_{0}^{p^m-1} + \sum_{i=0}^{3} F_i \, z_{0}^{a_i}z_{1}^{b_i} \cdots z_{p-1}^{b_i}$, 
where $a_i = p^m-1-c_{i}$ and 
$b_i = \frac{c_i}{p-1}$. 
From these identities and \eqref{ci} we get the desired expressions for $a_i$ and $b_i$, and thus the result follows.
\end{proof}

\section{Optimal curves}
Fix $q=p^m$ with $p$ prime. In this section we will consider Artin-Schreier curves of the form 
\begin{equation} \label{Crb} 
C_{R,\beta} \, : \quad  y^p-y = x R(x) +\beta x
\end{equation}
where $R(x)$ is any $p$-linearized polynomial over $\ff_q$ and $\beta \in \ff_q$.
A good treatment of Artin-Schreier curves can be found in Chapter 3 by Güneri-Özbudak in \cite{GS}. 
They are associated to the codes $\CC_{\LL,*}$ studied in Sections 3--5 which are defined by quadratic forms 
$Q_{R}(x) = \tr_{p^m/p}(xR(x))$, or similar ones, of Section 2.
Given a family $\LL$ of $p$-linearized polynomials, we define the family 
\begin{equation} \label{family curves} 
\Gamma_\LL = \{ C_{R,\beta} : R \in \LL, \beta \in \ff_q \}
\end{equation}
of curves $C_{R,\beta}$ as in \eqref{Crb}.

We begin by showing necessary and sufficient conditions for the family $\LL$ 
to contain optimal curves (maximal or minimal); that is, curves attaining equality in the Hasse-Weil bound (see Theorem 5.2.3 in \cite{HS})
$$|\#C(\ff_q)-(q+1)| \le 2g\sqrt q .$$

\begin{prop} \label{conditions}
Assume $\LL$ is an even rank family of $p$-linearized polynomials over $\ff_{p^m}$. 
Let $R\in \LL$ and let $r$ be the rank of the quadratic form $Q_R(x) = \tr_{p^m/p}(xR(x))$ and $v = v_p(\deg R)$ be the $p$-adic value of $\deg R$. 
Then, the family $\Gamma_\LL$ in \eqref{family curves} contains optimal curves, both maximal and minimal, if and only if there is some 
$R\in \LL$ with $$v=\tfrac{m-r}2.$$ 
In this case we have: 
\begin{enumerate}[(i)]
\item If $p$ is odd, the curve $C_{R,\beta} \in \Gamma_\LL$ is maximal (resp.\@ minimal) if and only if 
the codeword $c_R(\beta) = ( \tr_{p^m/p}(xR(x) + \beta x) )_{x \in \ff_{p^m}^*}$ in $\CC_{\LL,1}$ 
has weight $w(c_R(\beta)) = w_{2,1}$ (resp.\@ $w_{2,2}$) as in Table 3. 

\item If $p=2$, the curve $C_{R,\beta} \in \Gamma_\LL$ is maximal (resp.\@ minimal) if and only if either 
$w(c_R(\beta)) = w_{2,1}$ (resp.\@ $w_{2,2}$) or else $w(c_R(\beta)) = w_{3,2}$ (resp.\@ $w_{3,1}$) as in Table 3.
\end{enumerate}
\end{prop}

\begin{proof}
Consider the cyclic code 
$\CC_{\LL,1} = \{ c_R(\beta) = (\tr_{p^m/p}(xR(x)+\beta x))_{x\in\ff_{p^m}^*} : R\in \LL, \, \beta \in \ff_{p^m}\}$
as in \eqref{codigoCL0}.
The weight of the codeword $c_R(\beta)$ is related to the number of  $\mathbb{F}_{p^m}$-rational points of the  
curve $C_{R,\beta}$ given in \eqref{Crb}.
In fact, by Hilbert's Theorem 90 we have  
$$ \tr_{p^m/p}(xR(x) + \beta x) = 0 \qquad \Leftrightarrow \qquad y^p-y = x R(x) + \beta x \quad \text{ for some } y \in \ff_{p^m}.$$  
Since $C_{R,\beta}$ is a $p$-covering of $\mathbb{P}^1$, considering the point at infinity, we get 
\begin{equation} \label{rat points}
 \#C_{R,\beta}(\mathbb{F}_{p^m})  =  1 + p \, \#\{ x\in \ff_{p^m}: \tr_{p^m/p}(xR(x)+\beta x) = 0 \}  = p^{m+1} +1 - p \, w(c_R(\beta))
\end{equation} 
where the values of $w(c_R(\beta))$ are given in Table 3 with $q=p$. 

On the other hand, as an application of the Riemann-Hurwitz formula, the curve $y^p-y=f(x)$ with $f(x) \in \ff_q[x]$, has genus 
$g=\tfrac 12 (p-1)(\deg f)$ since the degree of $f$ is coprime with $p$ (see Example 2.4 in \cite{GS}). Hence, 
$C_{R,\beta}$ has genus 
$$ g(C_{R,\beta}) = \tfrac 12 (p-1)(\deg R) = \tfrac 12 (p-1) p^v$$ 
since $(\deg xR(x)+\beta x, p ) = (p^v+1,p) =1$, for $R \ne 0$. 
By the Hasse-Weil bound for curves we have that 
\begin{equation} \label{HWB}
p^m + 1 - (p-1) p^{v + \frac m2} \le \#C_{R}(\mathbb{F}_{p^m}) \le p^m + 1 + (p-1) p^{v + \frac m2}.
\end{equation}
	
To find maximal or minimal curves we need to ensure equality in the 
above inequalities; that is, by \eqref{rat points} and \eqref{HWB} we want that 
$$p^{m+1} - p\, w(c_R(\beta)) = p^m \pm (p-1) p^{v + \frac m2}, $$ 	
where the sign $+$ (resp.\@ $-$) corresponds to a maximal (resp.\@ minimal) curve.
Looking at Table 3 with $q=p$, we check that this could only happen if and only if 
$v = \frac{m-r}{2}$ and the weight $w(c_R(\beta))$ is $w_{2,1}$ (resp.\@ $w_{2,2}$) for a maximal (resp.\@ minimal) curve. 
Because of the presence of the factors $p-1$ in the weights, additional curves appear in the case $p=2$.
They correspond to $w(c_R(\beta))=w_{3,2}$ (resp.\@ $w_{3,1}$) for a maximal (resp.\@ minimal) curve.
Since the type of the quadratic form is fixed, only one of the two kind of maximal (or minimal) curves can appear if $p=2$. 
\end{proof}

Next, as an application of the spectra of cyclic codes, for a fixed number $\ell$ we consider the Artin-Schreier curves 
\begin{equation} \label{curve}
\begin{aligned}
C_{\gamma,\beta} :  & \ y^p-y = \gamma x^{p^{\ell}+1} + \beta x,     & \qquad \gamma\in \ff_{p^m}^*, \, \beta\in\mathbb{F}_{p^m} \\[1mm]
C_{\gamma_1,\gamma_2,\beta} : & \ y^p-y = \gamma_1 x^{p^{3\ell}+1} + \gamma_2 x^{p^{\ell}+1} + \beta x, & 
    \qquad \gamma_1,\gamma_2 \in \ff_{p^m}^*, \, \beta\in\mathbb{F}_{p^m}
\end{aligned}
\end{equation}
related to the codes $\CC_{\ell,1}$ and $\CC_{\{\ell,3\ell\},1}$ of Section 4 and 5, respectively; and we will show 
that the families $\{ C_{\gamma,\beta} \}$ 
and $\{ C_{ \gamma_1,\gamma_2, \beta } \}$  
contain several maximal and minimal curves. 
 
We begin by computing the $\ff_{p^m}$-rational points of the curves in the first family $\{ C_{ \gamma,\beta} \}$.

\begin{prop} \label{rationalpoints} 
Let $m$ and $\ell$ be positive integers such that $m_{\ell}$ is even and let $p$ a prime number. Consider the curve 
$C_{\gamma,\beta}$ as in \eqref{curve}
with $\gamma\in \ff_{p^m}^*$ and $\beta\in\mathbb{F}_{p^m}$. 
Fix $\gamma=\alpha^t$ and put $\varepsilon_{\ell} = (-1)^{\frac 12 m_{\ell}}$. Then, we have:
\begin{enumerate}[(a)]
\item If $p>2$, with $\tfrac 12 m_{\ell}$ even and $t\equiv 0\pmod{p^{(m,\ell)}+1}$, then 
$$\# C_{\gamma,\beta}(\mathbb{F}_{p^m}) = \left\{
\begin{array}{lll}
		p^m+1     													&& \text{for }  p^m-p^{m-2(m,\ell)}  \text{ $\beta$'s}, \\[1mm]
		p^m+1-p^{(m,\ell)}(p-1)p^{\frac m2} && \text{for }  p^{m-2(m,\ell)-1}-(p-1)p^{\frac{m}{2}-(m,\ell)-1} \ \text{ $\beta$'s}, \\[1mm]
		p^m+1+p^{(m,\ell)}p^{\frac m2} 			&& \text{for } (p^{m-2(m,\ell)-1}+ p^{\frac{m}{2}-(m,\ell)-1})(p-1) \text{ $\beta$'s}.
\end{array}	\right. $$
		
\item If $p>2$, with $\frac 12 m_{\ell}$ even and $t\not\equiv 0\pmod{p^{(m,\ell)}+1}$, then 
$$\# C_{\gamma,\beta}(\mathbb{F}_{p^m}) = \left\{
		\begin{array}{lll}
		p^m+1+(p-1)p^{\frac m2}    && \text{for } \ p^{m-1}+(p-1)p^{\frac m2}  \text{ $\beta$'s}, \\[1mm]
		p^m+1-p^{\frac m2} 				 && \text{for } \ (p-1)(p^{m-1}-p^{\frac{m}{2}-1}) \text{ $\beta$'s}. 
	\end{array} \right. $$

\item If $p>2$, with $\frac 12 m_{\ell}$ odd and $t\equiv \tfrac{p^{(m,\ell)}+1}{2}\pmod{p^{(m,\ell)}+1}$, then 
$$\# C_{\gamma,\beta}(\mathbb{F}_{p^m}) = \left\{
		\begin{array}{lll}
		p^m+1      													&& \text{for } p^m-p^{m-2(m,\ell)} \text{ $\beta$'s}, \\[1mm]
		p^m+1+p^{(m,\ell)}(p-1)p^{\frac m2} && \text{for } p^{m-2(m,\ell)-1}+(p-1)p^{\frac{m}{2}-(m,\ell)-1} \text{ $\beta$'s}, \\[1mm]
		p^m+1-p^{(m,\ell)}p^{\frac m2} 			&& \text{for } (p^{m-2(m,\ell)-1}- p^{\frac{m}{2}-(m,\ell)-1})(p-1) \text{ $\beta$'s}. 
		\end{array}		\right. $$

\item If $p>2$, with $\frac 12 m_{\ell}$ odd and $t\not\equiv \tfrac{p^{(m,\ell)}+1}{2}\pmod{p^{(m,\ell)}+1}$, then 
$$\# C_{\gamma,\beta}(\mathbb{F}_{p^m}) = \left\{		
\begin{array}{lll}
		p^m +1-(p-1)p^{\frac m2}    && \text{for } p^{m-1}-(p-1)p^{\tfrac{m}{2}-1}  \text{ $\beta$'s}, \\[1mm]
		p^m +1+p^{\frac m2} 				&& \text{for } (p^{m-1}+p^{\frac{m}{2}-1})(p-1) \text{ $\beta$'s}.
		\end{array}		\right. $$
		
\item If $p=2$ and $\gamma\in S_{2,m}(\ell) = \{ x^{2^\ell+1} : x\in \ff_{2^m}^*\}$ then 
$$\# C_{\gamma,\beta}(\mathbb{F}_{2^m}) = \left\{
		\begin{array}{lll}
		2^m +1     																			 && \text{for } 2^m-2^{m-2(m,\ell)} \text{ $\beta$'s}, \\[1mm]
		2^m +1 - \varepsilon_{\ell} \, 2^{\frac m2 + (m,\ell)} && 
		\text{for } 2^{m-2(m,\ell)-1}-\varepsilon_{\ell} \, 2^{\frac{m}{2}-(m,\ell)-1} \text{ $\beta$'s}, \\
		2^m+1+\varepsilon_{\ell} \, 2^{ \frac m2 + (m,\ell)} && \text{for } 2^{m-2(m,\ell)-1}+\varepsilon_{\ell} \, 2^{\frac{m}{2}-(m,\ell)-1} 
		\text{ $\beta$'s}.	\end{array} \right. $$
		
\item If $p=2$ and $\gamma \not\in S_{2,m}(\ell)$ then 
$$\# C_{\gamma,\beta}(\mathbb{F}_{2^m}) = \left\{
		\begin{array}{lll}
		2^m+1+\varepsilon_{\ell} \, 2^{\frac m2} && \text{for } 2^{m-1}+\varepsilon_{\ell} \, 2^{\frac{m}{2}-1} \text{ $\beta$'s}, \\[1mm]
		2^m+1-\varepsilon_{\ell} \, 2^{\frac m2} && \text{for } 2^{m-1}-\varepsilon_{\ell} \, 2^{\frac{m}{2}-1} \text{ $\beta$'s}. 
		\end{array} \right. $$
\end{enumerate}
\end{prop}

\begin{proof}
Consider the cyclic code 
$\CC_{\ell,1} = \{ c_{\gamma,\beta} = (\tr_{p^m/p}(\gamma x^{p^{\ell}+1}+\beta x))_{x\in \mathbb{F}_{p^m}^{*}} : 
\gamma, \beta \in \mathbb{F}_{p^m}\}$. 
By the same argument as in the previous proof, we have that
$$\#C_{\gamma,\beta}(\mathbb{F}_{p^m}) = p^{m+1}+1 - p \, w(c_{\gamma,\beta}).$$ 
Thus, the number of rational points of $C_{\gamma, \beta}$ are obtained from Tables~5--8, by using Theorems \ref{Thmpar} and 
\ref{Thmimpar}, by straightforward calculations.
\end{proof}

We now show the existence of optimal curves in the family $\{C_{\gamma,\beta}\}$. We will use Proposition \ref{conditions} to prove the existence of optimal curves and Proposition \ref{rationalpoints} to count the number of them.

\begin{thm} \label{teo minimal}
Let $p$ be a prime number. Let $m$ and $\ell$ non-negative integers with $\ell\mid m$ such that $m_{\ell}=\frac{m}{\ell}$ is even and 
$\gamma=\alpha^t\in\mathbb{F}_{p^m}$. Then, we have the following:
\begin{enumerate}[(a)]
\item Let $p$ be odd. Then, the curve $C_{\gamma,\beta}$ as in \eqref{curve} is \

\hspace{-.975cm} (i) minimal if $\frac{1}{2} m_{\ell}$ is even and $t\equiv 0\pmod{p^\ell+1}$, 
for $p^{m-2\ell-1}-(p-1)p^{\frac{m}{2}-\ell-1}$ elements $\beta$, \

\hspace{-.975cm} (ii) maximal if $\frac{1}{2} m_{\ell}$ is odd and $t\equiv \tfrac{p^\ell+1}{2} \pmod{p^\ell+1}$, 
for	$p^{m-2\ell-1}+(p-1)p^{\frac{m}{2}-\ell-1}$ elements $\beta$.
	
	\smallskip
\item Let $p=2$ and $\gamma\in S_{2,m}(\ell) = \{x^{2^\ell+1} : x\in\ff_{2^m}^*\}$. Then, \

\hspace{-.975cm} (i) there are $2^{m-2\ell-1}-2^{\frac{m}{2}-\ell-1}$ elements $\beta$ such that $C_{\gamma,\beta}$ is minimal and

\hspace{-.975cm} (ii) there are $2^{m-2\ell-1}+2^{\frac{m}{2}-\ell-1}$ elements $\beta$ such that $C_{\gamma,\beta}$ is maximal.
\end{enumerate}
\end{thm} 

\begin{proof}
Consider the family $\LL = \langle x^{p^\ell} \rangle$ of $p$-linearized polynomials over $\ff_{p^m}$, with $p$ prime. 
By Klapper's Theorems \ref{Thmpar} and \ref{Thmimpar}, $\LL$ is an even rank family. Thus, the family of curves $\Gamma_\LL$ in 
\eqref{Crb} is in fact the family $\{C_{\gamma,\beta}\}$ in \eqref{curve}.
Now, applying Proposition \ref{conditions}, by using Tables 3 and 7 and Theorems 
\ref{Thmpar} and \ref{Thmimpar}, we get the existence part of the statement. 
Finally, invoking Proposition~\ref{rationalpoints} we get the number of such optimal curves.
\end{proof}

\begin{exam}
Suppose that $p=2$, $m=4$ and $\ell=1$. Consider the curve  
$$ C_{\gamma,\beta} : y^2+y = \gamma x^3+\beta x, \qquad \gamma \in \ff_{16}^*, \quad \beta \in \ff_{16}$$ 
which is in particular an elliptic curve. Suppose that $\gamma \in S_{2,4}(1)$. Then, by Theorem \ref{teo minimal},
$C_{\gamma,\beta}$ is minimal for only one element $\beta$ and it is maximal for 3 elements $\beta$. 

If $\gamma= z^3$ for some $z\in \ff_{16}^*$ then, by the affine change of variable $u=zx$, 
the curve $C_{\gamma,\beta}$ turns out to be isomorphic to the curve 
$C_{1,\lambda} : y^2+y = u^3 + \lambda u$, where $\lambda = \beta z^{-1}$.
This curve is minimal ($9$ rational points) only for $\lambda=0$ and it is 
maximal ($25$ rational points) for $\lambda = 1, \alpha^{5}$ and $\alpha^{10}$. That is, 
$$y^2+y=u^3$$ 
is a minimal elliptic curve and 
$$y^2+y=u^3+u, \qquad y^2+y=u^3+\alpha^5 u, \qquad y^2+y=u^3+\alpha^{10} u$$ 
are maximal elliptic curves over $\ff_{16}$. 
\end{exam}

We now show that the family $\{ C_{\gamma_1, \gamma_2, \beta} \}$ contains optimal curves. 

\begin{prop}
Let $p$ be an odd prime and let $m$, $\ell$ be non-negative integers such that $\ell \mid m$, $m_{\ell} = \frac{m}{\ell}$ is even 
and $m>6\ell$. 
If $\frac 12 m_{\ell}$ is odd (resp.\@ even), the Artin-Schreier curve $C_{\gamma_1, \gamma_2, \beta}$ as in \eqref{curve} is 
maximal (resp.\@ minimal) for 
some $\gamma_1, \gamma_2 \in \ff_{p^m}^*$ and $\beta \in \ff_{p^m}$. 
\end{prop}

\begin{proof}
The family $\LL = \langle x^{p^\ell}, x^{3\ell} \rangle$ of $p$-linearized polynomials over $\ff_{p^m}$ has the even rank property, by Theorem \ref{ranks l3l}.
By Table 10 and Theorem \ref{spec cl3l}, we have that if $\frac 12 m_{\ell}$ is odd (resp.\@ even) then there exists 
$R \in \mathcal{L}$ with $\deg R =p^{3 \ell}$ and $Q_R$ of rank $r=m-6\ell$ and type $1$ (resp.\@ 3). 
Thus, we have that $v=\tfrac{m-r}2 = 3\ell$, where $v=v_{p}(\deg R)$, and the result follows directly from Proposition \ref{conditions}.
\end{proof}

\begin{exam}
Take $p$ an odd prime, $\ell=1$ and $m>6$ even. 
Then, the Artin-Schreier curve 
$y^p-y = \gamma_1 x^{p^3+1} + \gamma_2 x^{p+1} + \beta x$ is maximal in $\ff_{p^{4k}}$ and minimal in $\ff_{p^{4k+2}}$ for any 
$k \ge 2$, for at least one $\gamma_1,\gamma_2 \in \ff_{p^m}^*$ and $\beta \in \ff_{p^m}$, 
where $\ff_{p^m}$ stands for $\ff_{p^{4k}}$ or $\ff_{p^{4k+2}}$ depending on the case. For instance,  
$$y^3-y = \gamma_1 x^{28} + \gamma_2 x^{4} + \beta x$$ 
is maximal in $\ff_{3^8}= \ff_{6561}$ and minimal in $\ff_{3^{10}} = \ff_{59049}$ for at least one $\gamma_1,\gamma_2, \beta$ in the corresponding field. Similarly, 
$$y^5 - y = \gamma_1 x^{125} + \gamma_2 x^{6} + \beta x$$
is maximal in $\ff_{5^8} = \ff_{390625}$ and minimal in $\ff_{5^{10}} = \ff_{9765625}$ 
for some elements $\gamma_1,\gamma_2, \beta$ in the ground field. 
\end{exam}

\end{document}